\documentclass[11pt, oneside]{article} 
\usepackage[margin=2cm]{geometry}                		

\usepackage[parfill]{parskip}

\usepackage{graphicx}			
\usepackage{amsmath,stackengine}
\stackMath

\usepackage[utf8]{inputenc}
\usepackage[english]{babel} 
\usepackage{amsmath,amssymb,amsthm,bm} 
\usepackage[color=yellow]{todonotes}
\usepackage{mathrsfs}
\usepackage{url}	
\usepackage{esint}
\usepackage[colorlinks]{hyperref}
\usepackage{enumerate}
\usepackage{outlines}[enumerate]
\usepackage{framed,comment,enumerate}
\usepackage{upgreek}
\usepackage{pgfplots}
\usepackage{float}
\usepackage{tikz,tikz-cd}
\usepackage{rotating}
\usepackage{graphicx}
\usepackage{caption}
\usepackage{multicol}
\usepackage{cancel}
\usepackage{subcaption}
\usepackage[title]{appendix}
\usepackage{tikz-cd}
\usepackage{pgf, pgffor}
\usepackage{subcaption}
\usepackage{stmaryrd}

\DeclareFontFamily{U}{MnSymbolC}{}
\DeclareSymbolFont{MnSyC}{U}{MnSymbolC}{m}{n}
\DeclareFontShape{U}{MnSymbolC}{m}{n}{
    <-6>  MnSymbolC5
   <6-7>  MnSymbolC6
   <7-8>  MnSymbolC7
   <8-9>  MnSymbolC8
   <9-10> MnSymbolC9
  <10-12> MnSymbolC10
  <12->   MnSymbolC12}{}
\DeclareMathSymbol{\intprod}{\mathbin}{MnSyC}{'270}

\extrafloats{100}
\numberwithin{equation}{section}
\newtheorem{theorem}{Theorem}[section]

\newtheorem{remark}{Remark}[section]

\theoremstyle{definition}

\usepackage[color=yellow]{todonotes}

\makeatletter
\@tfor\next:=abcdefghijklmnopqrstuvwxyzABCDEFGHIJKLMNOPQRSTUVWXYZ\do{
\def\command@factory#1{%
\expandafter\def\csname b#1\endcsname{\mathbf{#1}}
\expandafter\def\csname fk#1\endcsname{\mathfrak{#1}}
\expandafter\def\csname bb#1\endcsname{\mathbb{#1}}
\expandafter\def\csname cl#1\endcsname{\mathcal{#1}}
\expandafter\def\csname bcl#1\endcsname{\mathbfcal{#1}}
}
\expandafter\command@factory\next
}

\newcommand{\ob}[1]{\overline{#1}}
\newcommand{\wt}[1]{\widetilde{#1}}

\newcommand{\mc}[1]{\mathcal{#1}}

\newcommand{\mcal}[1]{\mc{#1}}
\newcommand{\scp}[2]{\left<#1\,,\,#2\right>}

\newcommand{\ad}{\operatorname{ad}}

\newcommand{\dd}{\mathrm{d}}

\def\p{{\partial}}

\def\e{{\epsilon}}

\makeatother

\def\p{\partial}

\pgfplotsset{compat=1.16}

\begin{document}
\title{Stochastic multisymplectic PDEs and their structure-preserving numerical methods}
\author{Ruiao Hu$^1$\footnote{Corresponding author. Email: ruiao.hu15@imperial.ac.uk}~~and Linyu Peng$^2$\footnote{Email: l.peng@mech.keio.ac.jp} \vspace{0.4cm}\\
1. Department of Mathematics, Imperial College London, \\ London SW7 2AZ, UK\\
2. Department of Mechanical Engineering, Keio University,\\ Yokohama 223-8522, Japan}
\date{\today}
\maketitle

\begin{abstract}
    We construct stochastic multisymplectic systems by considering a stochastic extension to the variational formulation of multisymplectic partial differential equations proposed in [Hydon, {\it Proc. R. Soc. A}, 461, 1627--1637, 2005]. The stochastic variational principle implies the existence of stochastic $1$-form and $2$-form conservation laws, as well as conservation laws arising from continuous variational symmetries via a stochastic Noether's theorem. These results are the stochastic analogues of those found in deterministic variational principles. Furthermore, we develop stochastic structure-preserving collocation methods for this class of stochastic multisymplectic systems. These integrators possess a discrete analogue of the stochastic $2$-form conservation law and, in the case of linear systems, also guarantee discrete momentum conservation. The effectiveness of the proposed methods is demonstrated through their application to stochastic nonlinear Schr\"odinger equations featuring either stochastic transport or stochastic dispersion.
\end{abstract}

\tableofcontents

\section{Introduction}
The inclusion of stochasticity into dynamical systems is natural in modelling the effects of unresolvable fluctuations or unknown physical processes. One important application is in weather/climate forecasting, where stochastic perturbations are introduced to increase the ensemble variance \cite{BTB2015,BMP1999}. However, when the physical dynamics systems possess inherent structures, and in turn, conservation laws, special types of structure-preserving stochasticity are required such that the underlying structures are preserved in the presence of noise.

For Hamiltonian dynamics in finite dimensions, canonical Hamilton's equations possess conservation laws that include the total energy (Hamiltonian), symplectic form and Liouville form \cite{Holm2011,MR1999}.
In addition to conservation laws, when the system possesses variational symmetries, reduction by symmetry can be performed to obtain Lie--Poisson equations that are equivalent to canonical Hamilton's equations \cite{HMR1998, MMOPR2007, MRW1984a}. The dynamics of the reduced system lies on coadjoint orbits and they typically have Casimir functions that are invariant under the motion.
One well-studied class of structure-preserving stochastic dynamics is the stochastic canonical Hamilton's equations, see, e.g., \cite{Bismut1982, BRO2009, LCO2008}, where the conservation of symplectic form and Liouville Theorem are preserved under the stochastic motion. However, the conservation of energy is lost.
In addition to the preservation of conservation laws, when the deterministic system and its stochastic perturbations possess variational symmetries, reduction by symmetry can be performed to obtain equivalent stochastic Lie--Poisson equations.
In the pioneering work of Holm \cite{Holm2015}, stochasticity of Hamiltonian type is derived for applications to ideal fluid dynamics which takes the form of an infinite-dimensional Lie--Poisson equation defined on $\mathfrak{X}^*$, the space of $1$-form densities. In this setting, the stochastic perturbations take the form of stochastic transport, exhibiting remarkable analytical properties.

The aim of this paper is twofold. The first is to provide an extension of the stochastic canonical Hamilton's equations to infinite dimensions by modifying the variational multisymplectic framework of \cite{Bridges1997, BHL2010, Hydon2005} to include stochasticity. By working in the stochastic multisymplectic framework, the resulting stochastic multisymplectic partial differential equations (PDEs), also referred to as Hamiltonian PDEs, inherently possess stochastic analogues of the deterministic conservation laws.
The second aim of this paper is to construct structure-preserving methods such that (at least some of) the stochastic conservation laws are preserved in numerical simulations. In the deterministic setting, these methods are known as multisymplectic methods first introduced in \cite{BR2001} which has seen extensions to the stochastic setting, see, e.g., \cite{CHJ2023, JWH2013}.


\paragraph{Main contents of the paper.} 
We summarise the main contributions and contents of this paper as follows.
\begin{itemize}
    \item In Section \ref{sec:stochastic multisymplectic cont}, after a brief review of deterministic multisymplectic systems and their associated conservation laws, we construct a general class of stochastic multisymplectic systems by considering appropriate stochastic variational principles in Section \ref{subsec:stoch multi}. These stochastic dynamics possess stochastic analogues of the deterministic multisymplectic $1$-form and $2$-form conservation laws, and they are sufficiently general to incorporate existing stochastic multisymplectic systems. To demonstrate the flexibility of the construction, we present the multisymplectic formulation of stochastic transport noise and stochastic nonlinear Schr\"odinger (NLS) equations in Sections \ref{sec:SALT} and \ref{subsec:NLS}, respectively.
    \item In Section \ref{sec:stochastic multisymplectic discrete}, we construct a general family of stochastic structure-preserving numerical schemes based on stochastic Runge--Kutta methods that preserve exactly, a discrete version of the continuous stochastic symplectic form conservation law. For quadratic Hamiltonians, we show that the methods constructed can additionally preserve exactly, a discrete version of the stochastic momentum conservation law. In Section \ref{subsec:NLS numerics}, we evaluate the behaviour of the numerical schemes applied to stochastic NLS equations obtained in Section \ref{subsec:NLS}.
    \item Section \ref{sec:conclusion} contains concluding remarks and a discussion on directions for future work.
\end{itemize}

\section{Stochastic multisymplectic systems} \label{sec:stochastic multisymplectic cont}
\subsection{Review of deterministic multisymplectic systems}
We brief review the formulations of multisymplectic systems and their associated conservation laws following \cite{BHL2010,CHH2007,Hydon2005}. A system of PDEs is said to be multisymplectic if 
it can be derived from a variational principle with the Lagrangian $L(x,[z])$ 
of the following form
\begin{align}
  L({x},[{z}]) = L^i_\alpha({x},{z}) z^\alpha_{,i} - H({x},{z})\,,\label{eq:det multi lag}
\end{align}
where ${x}=(x^0, x^1,\ldots,x^n)\in\mathbb{R}^{n+1}$ is the independent (base space) variable with the identification $x^0 = t$, ${z}=(z^1,z^2,\ldots,z^m)\in\mathbb{R}^m$ is the dependent variable, and $H({x},{z})$ is the Hamiltonian of the system. Furthermore, 
total derivative $D_i$ with respect to $x^i$ is also denoted by the subscript $i$ after a comma, $[{z}]$ denotes ${z}$ and finitely many of its derivatives, and the Einstein summation convention is assumed unless otherwise stated. Hence, the Lagrangian for a (deterministic) multisymplectic system is affine in the first derivatives of the dependent variables.



To derive the Euler--Lagrange equations associated with the Lagrangian $L(x,[z])$, we consider the corresponding  variational principle,
\begin{align*}
    0 = \delta \mathcal{S}[z] = \delta \int L({x},[{z}])\, \dd vol\,,
\end{align*}
where $\dd vol$ is the volume form on the base space locally defined by $\dd vol := \dd t \wedge \dd x^1\wedge  \cdots \wedge \dd x^n$ and the variations $\delta z$ are assumed to be arbitrary and vanishing at boundaries. This results the Euler--Lagrange equations
\begin{align}
   K^i_{\alpha\beta}({x},{z})z^\beta_{,i} - \frac{\p L^i_\alpha}{\p x^i} = \frac{\p H}{\p z^\alpha}\,, \quad \text{where} \quad K^i_{\alpha\beta} := \frac{\p L^i_\beta}{\partial z^\alpha} - \frac{\p L^i_\alpha}{\partial z^\beta}\,, \label{eq:det multi symplectic}
\end{align}
and $K^i_{\alpha\beta}$ are known as the multisymplectic structure matrices. By construction, $K^i_{\alpha\beta}$ are antisymmetric for all $i$ and we can construct closed $2$-forms $\kappa^i$ by 
\begin{align}
    \kappa^i = \frac{1}{2}K^i_{\alpha\beta}({x},{z})\,\dd z^\alpha\wedge \dd z^\beta,\quad i=0,1,\ldots,n\,,
\end{align}
which satisfy
\begin{align}
    \kappa^i_{,i} = 0\,.\label{eq:det structural conservation laws}
\end{align}
This is known as the structural conservation law \cite{Bridges1997n}. Via the Poincar\'e Lemma for the variational bicomplex (see, e.g., \cite{Anderson1992, BHL2010, PH2023}), we can obtain 
\begin{equation}\label{eq:cls}
    \left(L_\beta^{j}(z)z_{,i}^\beta\right)_{,j}-\left(L^{j}_\beta(z)z^\beta_{,j}-H\right)_{,i} = \frac{\partial H}{\partial x^{i}}-\frac{\partial L_\beta^{j}}{\partial x^{i}}z^\beta_{,j}\, , \quad i = 0,1,\ldots,n\, .
\end{equation}
If the Lagrangian does not depend on some particular components of the independent variables, then conservation laws can be obtained as the RHS for these components vanishes. 
This also generalises the conservation laws obtained in \cite{Hydon2005}; namely, when  the Lagrangian is independent from all $x^i$ (which is often true), the RHS of \eqref{eq:cls} always vanishes.
In this special case, the Euler--Lagrange equations \eqref{eq:det multi symplectic} become
\begin{align}
   K^i_{\alpha\beta}({z})z^\beta_{,i}  = \frac{\p H}{\p z^\alpha}\,,
\end{align}
and now the $1$-form $\dd L(x,[z]) = \dd (L^i_\alpha({z})z^\alpha_{,i} - H({z}))$ satisfies the $1$-form quasi-conservation law
\begin{align}
     (L^i_\alpha({z}) \dd z^\alpha)_{,i} = \dd \left(L^i_\alpha({z})z^\alpha_{i} - H({z})\right) = \dd  L(x,[z])\,.
     \label{eq:det one form conservations laws}
\end{align}
Specially, the $j=0$ case in equation \eqref{eq:cls} gives the energy conservation law and the other cases give momentum conservation laws \cite{BR2001}.
In fact, these conservation laws belong to a broader class that results from Noether's theorem, associated with continuous variational symmetries of the system. As it was shown in, e.g., \cite{Hydon2005, Olver1993}, each variational symmetry of the  multisymplectic system \eqref{eq:det multi symplectic} corresponds to a (prolonged) vector field of the form
\begin{align}\label{eq:symgen}
    X = Q^\alpha(x, [z]) \frac{\p}{\p z^\alpha} + Q^\alpha_{,i}(x,[z])\frac{\p}{\p z^\alpha_{,i}}\, , 
\end{align}
which is commonly referred to as the evolutionary representative of a generalised symmetry, with the $m$-tuple $Q$ being its characteristic.
We say that $X$ generates (divergence) variational symmetries of the Lagrangian $L(x,[z])$ if 
\begin{align}
    X L(x,[z]) = i_X \dd L(x,[z]) = B^i_{,i}
    \label{eq:X def}
\end{align}
holds for some $(n+1)$-tuple $(B^0,B^1,\ldots, B^n)$ with $B^i = B^i(x, [z])$, for $i = 0,1,\ldots, n$. Noether's theorem implies that the $X$ yields the conservation law
\begin{align}
    \left(L^i_\alpha Q^\alpha - B^i\right)_{,i} = 0\,,\label{eq:det noether cls}
\end{align}
which can be seen by combining equations \eqref{eq:det one form conservations laws} and \eqref{eq:X def}. For translational symmetry in the independent variable $x^i$, the vector field of the symmetry is given by $X_i := z^\beta_{,i}\p_{z^\beta} + z^\beta_{,ij}\p_{z^\beta_{,j}}$ and $X_i L = L_{,i}$. Then, Noether's theorem yields the conservation law
\begin{align}
    \left(L^j_\beta z^\beta_{,i} - \left(L^{k}_\beta z^\beta_{,k}-H\right)\delta^j_i\right)_{,j} = 0\, ,
\end{align}
which is equivalent to equation \eqref{eq:cls} with vanishing RHS.

\subsection{Stochastic multisymplectic systems}\label{subsec:stoch multi}
In the presence of stochasticity, the time component of the base manifold becomes special as the time derivative has to be considered in an integral form. To keep the notation simple, we keep the notation $x^i$ to mean the $i$-th base coordinate for $i = 1,2,\ldots, n$ and $x^0 = t$ to mean the time coordinate. The volume form of the base space is given by $\dd vol^s = \dd x^1\wedge \dd x^2\wedge\ldots \wedge \dd x^n$ which excludes the time component. We retain the notation $[{z}]$ to instead denote ${z}$ and finitely many of its derivatives with respect to $x^i$ for $i = 1,2,\ldots, n$, excluding $x^0 = t$. In the Einstein summation notation for the base space variables, the summations exclude the time coordinates such that the summation range is between $1$ and $n$ inclusive. 

Let $(\Omega,\mathcal{F},\bbP)$ denote a probability space supporting a one-dimensional Brownian motion $W$. Let $\bbF$ denote the filtration generated by $W$ and we assume that all $m$ components of the dependent variable ${z} = (z^1,z^2,\ldots,z^m)$ are $\bbF$-adapted. For extension to higher-dimensional Brownian motion, we refer to Remark \ref{rmk:multi BM}.

We consider a stochastic action in the form of 
\begin{align}
\mathcal{S}[z]= \int \int_{t_0}^{t_1} \left[L^0_\alpha({z}) d z^\alpha + L^i_\alpha({z}) z^\alpha_{,i}\,dt + \wt{L}^i_\alpha({z}) z^\alpha_{,i}\circ dW_t  - H({z}) \,dt - \wt{H}({z}) \circ dW_t\,\right] \dd vol^s\,, \label{eq:stoch ms action}
\end{align}
where the notation $dz^\alpha$ is the stochastic time increment of $z^\alpha$ and the notation $\circ dW_t$ denotes the stochastic integrals in the Stratonovich sense such that the standard rules of calculus, e.g., product rule and chain rule, applies. The action \eqref{eq:stoch ms action} here is reminiscent of the stochastic phase space variational principle in the ordinary differential equation case which yields stochastic canonical Hamilton's equations (e.g., \cite{Bismut1982, LCO2008}). 
Taking variations, we get
\begin{align*}
    \begin{split}
        0 = \delta \mathcal{S}[z] &= \int\int_{t_0}^{t_1} \Bigg[ \frac{\p L^0_\beta}{\p z^\alpha}d z^\beta \delta z^\alpha + L^0_\beta d \delta z^\beta + \frac{\p L^i_\beta}{\p z^\alpha}z^\beta_{,i}\delta z^\alpha\,dt + L^i_\beta \delta z^\beta_{,i}\,dt  \\
        & \qquad \qquad \qquad  + \frac{\p \wt{L}^i_\beta}{\p z^\alpha}z^\beta_{,i}\delta z^\alpha\circ dW_t + \wt{L}^i_\beta\delta z^\beta_{,i}\circ dW_t - \frac{\p H}{\p z^\beta}\delta z^\beta\,dt - \frac{\p \wt{H}}{\p z^\beta}\delta z^\beta\circ dW_t\, \Bigg]\dd vol^s\\
        &= \int \int_{t_0}^{t_1} \Bigg[\frac{\p L^0_\beta}{\p z^\alpha}d z^\beta \delta z^\alpha - \frac{\p L^0_\beta}{\p z^\alpha} d z^\alpha \delta z^\beta + \frac{\p L^i_\beta}{\p z^\alpha}z^\beta_{,i}\delta z^\alpha\,dt - \frac{\p L^i_\beta}{\p z^\alpha} z^\alpha_{,i} \delta z^\beta\,dt+ \frac{\p \wt{L}^i_\beta}{\p z^\alpha}z^\beta_{,i}\delta z^\alpha\circ dW_t \\
        & \qquad \qquad \qquad   - \frac{\p \wt{L}^i_\beta}{\p z^\alpha} z^\alpha_{, i}\delta z^\beta\circ dW_t - \frac{\p H}{\p z^\beta}\delta z^\beta\,dt - \frac{\p \wt{H}}{\p z^\beta}\delta z^\beta\circ dW_t \,\Bigg] \dd vol^s +\text{BTs}\\
        &=\int \int_{t_0}^{t_1} \Bigg[ \left(\frac{\p L^0_\beta}{\p z^\alpha} - \frac{\p L^0_\alpha}{\p z^\beta} \right) d z^\beta + \left(\frac{\p L^i_\beta}{\p z^\alpha} - \frac{\p L^i_\alpha}{\p z^\beta}\right)z^\beta_{,i}\,dt + \left(\frac{\p \wt{L}^i_\beta}{\p z^\alpha} - \frac{\p \wt{L}^i_\alpha}{\p z^\beta}\right)z^\beta_{,i}\circ dW_t \\
        & \hspace{6em} - \frac{\p H}{\p z^\alpha}\,dt - \frac{\p \wt{H}}{\p z^\alpha}\circ dW_t \,\Bigg]\delta z^\alpha\, \dd vol^s + \textrm{BTs} \,,
    \end{split}
\end{align*}
where $\textrm{BTs}$ denotes the boundary terms. Choosing appropriate boundary conditions such that $\textrm{BTs}$ vanish, we obtain the stochastic multisymplectic Hamilton's equations from stochastic fundamental lemma of calculus of variations \cite{ST2023},
\begin{align}
     \left(\frac{\p L^0_\beta}{\p z^\alpha} - \frac{\p L^0_\alpha}{\p z^\beta} \right)d z^\beta + \left[\left(\frac{\p L^i_\beta}{\p z^\alpha} - \frac{\p L^i_\alpha}{\p z^\beta}\right)z^\beta_{,i} - \frac{\p H}{\p z^\alpha}\right]\,dt + \left[\left(\frac{\p \wt{L}^i_\beta}{\p z^\alpha} - \frac{\p \wt{L}^i_\alpha}{\p z^\beta}\right)z^\beta_{,i} - \frac{\p \wt{H}}{\p z^\alpha}\right]\circ dW_t  = 0 \,.
\end{align}
This can be arranged into a form that is reminiscent of the multisymplectic PDEs introduced in \cite{Bridges1997}, 
\begin{align}
    M_{\alpha\beta}d z^\beta + K^i_{\alpha\beta}z^\beta_{,i}\,dt + \wt{K}^i_{\alpha\beta}z^\beta_{,i}\circ dW_t = \frac{\p H}{\p z^\alpha}\,dt + \frac{\p \wt{H}}{\p z^\alpha}\circ dW_t\,,\label{eq:stoch multi symplectic}
\end{align}
where the skew-symmetric matrices $K^i_{\alpha\beta}$, $\wt{K}^i_{\alpha\beta}$, and $M_{\alpha\beta}$ are defined by
\begin{align}\label{eq:mkk}
      \quad M_{\alpha\beta} = \frac{\p L^0_\beta}{\p z^\alpha} - \frac{\p L^0_\alpha}{\p z^\beta}\, , \quad K^i_{\alpha\beta} = \frac{\p L^i_\beta}{\p z^\alpha} - \frac{\p L^i_\alpha}{\p z^\beta}\,, \quad \text{and} \quad  \wt{K}^i_{\alpha\beta} = \frac{\p \wt{L}^i_\beta}{\p z^\alpha} - \frac{\p \wt{L}^i_\alpha}{\p z^\beta}\,,
\end{align}
for $i=1,2,\ldots,n$, and  $\alpha,\beta=1,2,\ldots,m$.
We remark on some previous work on stochastic multisymplectic PDEs, notably starting in \cite{JWH2013} and subsequent publications, e.g., \cite{CHJ2023}. In the cases previously considered, stochasticity only appeared in the term of $\wt{H}$ as a perturbation to the Hamiltonian. One of the contributions in this work is the introduction of stochasticity multisymplectic matrices $\wt{K}^i_{\alpha\beta}$ for each of the base space variables which will play a role in the $1$-form quasi-conservation laws and the $2$-form conservation laws.

\begin{remark}[Constrained stochastic degrees of freedom]
    We remark that in practice, the skew-symmetric matrix $M_{ij}$ typically contains null rows. Suppose that the $\alpha$-th row $M_{\alpha\beta} = 0$ for $\alpha \in I$ where $I \subset [0,1,\ldots,m]$. Then, by the Doob--Meyer semimartingale decomposition theorem \cite{meyer1963decomposition}, we have
    \begin{align*}
        K^i_{\alpha\beta}z^\beta_{,i}\, dt = \frac{\p H}{\p z^\alpha}\,dt \,,\quad \text{and}\quad \wt{K}^i_{\alpha\beta}z^\beta_{,i} \circ dW_t = \frac{\p \wt{H}}{\p z^\alpha} \circ dW_t\,, \quad \text{for}\quad \alpha \in I\,.
    \end{align*}
    These relations are constraints between the spatial symplectic structures and the variational derivatives of the Hamiltonians. In the $\circ dW$ component, these equations enforces constraints between the stochastic degrees of freedom, $\wt{K}^i_{\alpha\beta}$ and $\wt{H}$. 
\end{remark}

Following the deterministic case, we consider the $1$-forms $L^0_\alpha \dd z^\alpha$, $L^i_\alpha \dd z^\alpha$ and $\wt{L}^i_\alpha \dd z^\alpha$. Through direct calculation, we have a stochastic version of $1$-form quasi-conservation law
\begin{align}
    d(L^0_\alpha \dd z^\alpha) + (L^i_\alpha \dd z^\alpha \,dt + \wt{L}^i_\alpha \dd z^\alpha \circ dW_t)_{, i} = \dd(L^0_\alpha d z^\alpha + L^i_\alpha z^\alpha_{,i} \,dt + \wt{L}^i_\alpha z^\alpha_{,i} \circ dW_t - H\,dt - \wt{H}\circ dW_t)\,, \label{eq:stochastic 1-form conservation law}
\end{align}
which extends the deterministic case \eqref{eq:det one form conservations laws}.
Taking the exterior derivative of \eqref{eq:stochastic 1-form conservation law} gives the structural conservation law
\begin{align}
    d \dd (L^0_\alpha \dd z^\alpha) + \dd(L^i_\alpha \dd z^\alpha \,dt + \wt{L}^i_\alpha \dd z^\alpha \circ dW_t)_{, i} = 0\,.
\end{align}
By introducing the exact $2$-forms 
\begin{equation}
\begin{aligned}
    \omega &= \dd (L^0_\alpha \dd z^\alpha) = \frac{1}{2}M_{\alpha\beta}\dd z^\alpha\wedge \dd z^\beta\,  , \\
     \kappa^i &= \dd (L^i_\alpha \dd z^\alpha)= \frac{1}{2}K^i_{\alpha\beta}\dd z^\alpha\wedge \dd z^\beta\,  , \\ \wt{\kappa}^i &= \dd(\wt{L}^i_\alpha \dd z^\alpha)= \frac{1}{2}\wt{K}_{\alpha\beta}^{i}\dd z^\alpha\wedge \dd z^\beta\,, 
\end{aligned}
\end{equation}
the structural conservation law can be written in the form
\begin{align}
    d \omega + \kappa^i_{,i}\,dt + \wt{\kappa}^i_{, i}\circ dW_t = 0\,. \label{eq:stoch structural conservation laws}
\end{align}
This is the stochastic generalisation of equation \eqref{eq:det structural conservation laws}.
Physical conservation laws can be derived by projecting the $1$-form quasi-conservation laws \eqref{eq:stochastic 1-form conservation law} onto the base space variables. For the spatial components, this is done by replacing each $1$-from $\dd z^{\alpha}$ by $z^{\alpha}_{,i} \dd x^i$ and evaluating at each $i$. Then, we have
\begin{align}
\begin{split}
    d (L^0_\alpha z^\alpha_{,j}) + (L^i_\alpha z^\alpha_{,j}\,dt + \wt{L}^i_\alpha z^\alpha_{,j}\circ dW_t)_{,i} = (L^0_\alpha d z^\alpha + L^i_\alpha z^\alpha_{,i}\,dt + \wt{L}^i_\alpha z^\alpha_{,i}\circ dW_t - H\,dt - \wt{H} \circ dW_t)_{, j} \,,\label{eq:stoch conservations laws}
\end{split}
\end{align}
which can be rewritten as
\begin{equation}
      d (L^0_\alpha z^\alpha_{,j})- (L^0_\alpha d z^\alpha)_{,j}+ \left((L^i_\alpha z^\alpha_{,j})_{,i}- (L^i_\alpha z^\alpha_{,i} -H )_{,j} \right)dt + \left( (\wt{L}^i_\alpha z^\alpha_{,j})_{,i} - (   \wt{L}^i_\alpha z^\alpha_{,i} - \wt{H} )_{, j}\right) \circ dW_t=0 \, .\label{eq:stoch 1-form cls pullback} 
\end{equation}

\begin{remark}[Loss of energy conservation]
    We note that energy conservation no longer holds in the presence of stochasticity. Additionally, we cannot pullback the $1$-form quasi-conservation law \eqref{eq:stochastic 1-form conservation law} to the $t$ component of the base space since the quantity $z_{,0} = \p_t z$ is not well defined in the stochastic case.
\end{remark}

To consider continuous variational symmetries and Noether's conservation laws in the stochastic setting, we modify the condition on when a (prolonged) vector field generates a variational symmetry of the stochastic multisymplectic system \eqref{eq:stoch multi symplectic}. We rewrite the action \eqref{eq:stoch ms action} to its finite variation and martingale part such that
\begin{align}\label{eq:modvar}
    \mathcal{S}[z] = \int \int_{t_0}^{t_1} \left[L\,dt + \wt{L}\circ dW_t \right]\dd vol^s\,,
\end{align}
where 
\begin{align}
\begin{split}
    &L := L^0_\alpha({z}) \eta^\alpha + L^i_\alpha({z}) z^\alpha_{,i} - H({z}) \,, \quad \wt{L} := L^0_\alpha({z}) \wt{\eta}^\alpha + \wt{L}^i_\alpha({z}) z^\alpha_{,i} - \wt{H}({z})\,,
\end{split}
\end{align}
and
\begin{equation}
d z^\alpha = \eta^\alpha\,dt + \wt{\eta}^\alpha\circ dW_t\,.
\end{equation}
Here, we have implicitly assumed a semimartingale decomposition for the dependent variable $z$ to have $\eta=\eta(x,t)$ and $\wt{\eta}=\wt{\eta}(x,t)$ as functions of $(x,t)$ for a given section $z(x,t)$. Using the preceding notations, the stochastic $1$-form quasi-conservation law \eqref{eq:stochastic 1-form conservation law} can be simply expressed as
\begin{align}
    d(L^0_\alpha \dd z^\alpha) + (L^i_\alpha \dd z^\alpha \,dt + \wt{L}^i_\alpha \dd z^\alpha \circ dW_t)_{, i} = \dd(L\,dt + \wt{L}\circ dW_t)\,. \label{eq:stoch 1-form cls L wtL}
\end{align}
We remark that the semimartingale decomposition of $z$ should be interpreted as a definition rather than a constraint. In the subsequent computations, $dz^\alpha$ will be replaced by  $\eta^\alpha \,dt + \wt{\eta} \circ dW_t$ before considering actions of vector fields.

As $\eta$ and $\wt{\eta}$ replace the role of $``\p_t z"$ (which is undefined in the stochastic setting) in the construction of jet prolonged vector fields that generate variational symmetries,  a new form of the symmetry generator is required, different from that of \eqref{eq:symgen}. In this case, consider the evolutionary generator
\begin{equation}
\begin{aligned}
    X & = Q^\alpha(x, t, [z],[\eta],[\wt{\eta}]) \frac{\p}{\p z^\alpha} + Q^\alpha_{,i}(x, t, [z],[\eta],[\wt{\eta}])\frac{\p}{\p z^\alpha_{,i}} \\
    &\qquad \quad + P^\alpha(x,t,[z],[\eta],[\wt{\eta}])\frac{\p}{\p \eta^\alpha} + \wt{P}^\alpha(x,t,[z],[\eta],[\wt{\eta}])\frac{\p}{\p \wt{\eta}^\alpha}\,,
\end{aligned}
\end{equation}
where the relation of $Q$, $P$ and $\wt{P}$ is yet  to be defined. We say that $X$ generates a (divergence) variational symmetry of the stochastic action \eqref{eq:modvar} if 
\begin{align}
    X{L} = i_X \dd {L} = B^i_{,i} \quad \textrm{and}\quad X \wt{L} = i_X \dd \wt{L} = \wt{B}^i_{,i} \label{eq:X def stoch}
\end{align}
hold for some $n$-tuples $B^i = B^i(x,t, [z],[\eta],[\wt{\eta}])$ and $\wt{B}^i = \wt{B}^i(x,t,[z],[\eta],[\wt{\eta}])$ for $i = 1,2,\ldots, n$. Utilising the semimartingale decomposition of $L^0_\alpha$,
\begin{align*}
    d L^0_\alpha = F_\alpha^0\,dt + \wt{F}^0_\alpha\circ dW_t\,,
\end{align*}
for some $F^0_\alpha$ and $\wt{F}^0_\alpha$, the stochastic $1$-form quasi-conservation law \eqref{eq:stoch 1-form cls L wtL} can be expressed as
\begin{align}
    \begin{split}
        \left(F^0_\alpha\dd z^\alpha + L^0_\alpha \dd \eta^\alpha + (L^i_\alpha \dd z^\alpha)_{, i}\right)\,dt + 
        \left(\wt{F}^0_\alpha\dd z^\alpha + L^0_\alpha \dd \wt{\eta}^\alpha + (\wt{L}^i_\alpha \dd z^\alpha)_{, i}\right)\circ dW_t =\dd L\,dt + \dd \wt{L}\circ dW_t\,.
    \end{split}
\end{align}
Taking the interior product of $X$ and using the symmetry definition \eqref{eq:X def stoch}, we obtain the following relation
\begin{align}
    \left(F^0_\alpha Q^\alpha + L^0_\alpha P^\alpha + (L^i_\alpha Q^\alpha)_{, i} \right)\,dt + \left(\wt{F}^0_\alpha Q^\alpha + L^0_\alpha \wt{P}^\alpha + (\wt{L}^i_\alpha Q^\alpha)_{, i} \right)\circ dW_t = B^i_{,i}\,dt + \wt{B}^i_{,i} \circ dW_t \,. \label{eq:stoch cls noether broken}
\end{align}
When the following stochastic differential relation between $Q$, $P$ and $\wt{P}$ holds,
\begin{align}
    d Q^\alpha = P^\alpha \,dt + \wt{P}^\alpha \circ dW_t\,, \label{eq:compatbility cond}
\end{align}
equation \eqref{eq:stoch cls noether broken} is equivalent to the conservation law
\begin{align}
    d(L^0_\alpha Q^\alpha) + \left(L^i_\alpha Q^\alpha\,dt + \wt{L}^i_\alpha Q^\alpha\circ dW_t - B^i\,dt - \wt{B}^i\circ dW_t\right)_{,i} = 0\,,\label{eq:stoch cls noether}
\end{align}
which is the stochastic generalisation of the deterministic conservation law \eqref{eq:det noether cls}. 
This establishes a stochastic version of Noether's theorem, connecting variational symmetries with conservation laws of stochastic multisymplectic PDEs.

\begin{remark}
    The relation \eqref{eq:compatbility cond} can be interpreted as the condition that
    \begin{align*}
        d(i_X \dd z^\alpha) = i_X \dd\! \left( dz^\alpha\right)\,,
    \end{align*}
    which is analogous to the following relation in the deterministic case:
    \begin{equation*}
      D_i(i_X  \omega) = i_X  (D_i \omega)\,,
    \end{equation*}
    where $\omega$ is a differential form over the prolonged jet space, $X$ is an evolutionary vector field, and $D_i$ denotes the total derivative of the $i$-th independent variable. An interpretation of the deterministic case within the framework of the variational bicomplex is available in, e.g., \cite{Anderson1992, BHL2010}.
\end{remark}

As the stochastic multisymplectic systems \eqref{eq:stoch multi symplectic} and the corresponding   action \eqref{eq:stoch ms action} are both translational invariant in the spatial base variables,  we consider the evolutionary generator of translational symmetry  
\begin{align*}
    X_j=z^\alpha_{,j} \frac{\p}{\p z^\alpha}+z^\alpha_{,ji}\frac{\p}{\p z^\alpha_{,i}} + \eta_{,j}^\alpha\frac{\p}{\p \eta^\alpha} + \wt{\eta}_{,j}^\alpha\frac{\p}{\p \wt{\eta}^\alpha}
\end{align*}
for each $j \in \{1,2,\ldots, n\}$. Then, we have $X_j L = L_{,j}$ and $X_j \wt{L} = \wt{L}_{,j}$, and the relation \eqref{eq:compatbility cond} is also satisfied. Consequently, we obtain the conservation law \eqref{eq:stoch 1-form cls pullback} from the general one \eqref{eq:stoch cls noether}, similar to the deterministic case.
\begin{remark}[Driving multidimensional Brownian motion]\label{rmk:multi BM}
    In the preceding exposition, we have considered stochastic multisymplectic equations derived from the stochastic action \eqref{eq:stoch ms action} that is driven by a $1$-dimensional Brownian motion $W$. The extension to stochastic multisymplectic equations with driving multidimensional Brownian motion can be derived in an analogous fashion by considering the following action
    \begin{align*}
        \mathcal{S}[z] = \int \int_{t_0}^{t_1} \left[L^0_\alpha({z}) d z^\alpha + L^i_\alpha({z}) z^\alpha_{,i}\,dt + \wt{L}^{i}_{\alpha\beta}({z}) z^\alpha_{,i}\circ dW^{\beta}_t  - H({z}) \,dt - \wt{H}_\beta({z}) \circ dW^\beta_t\,\right] \dd vol^s\,.
    \end{align*}
    Here, $W^\beta_t$ are the components of the $N$-dimensional Brownian motion ${W} = (W^1,W^2,\ldots,W^N)$, where again we assume that the dependent variable ${z} = (z^1,z^2,\ldots,z^m)$ are adapted to the filtration generated by ${W}$. Hamilton's principle thus yields the stochastic multisymplectic system
    \begin{align*}
        M_{\alpha\beta}d z^\beta + K^i_{\alpha\beta}z^\beta_{,i}\,dt + \wt{K}^{i}_{\alpha\beta\gamma}z^\beta_{,i}\circ dW^\gamma_t = \frac{\p H}{\p z^\alpha}\,dt + \frac{\p \wt{H}_\gamma}{\p z^\alpha}\circ dW^\gamma_t\,,
    \end{align*}
    where $$\wt{K}^{i}_{\alpha\beta\gamma} := \frac{\partial \wt{L}^{i}_{\beta\gamma}}{\partial z^\alpha} - \frac{\partial \wt{L}^{i}_{\alpha\gamma}}{\partial z^\beta}\, ,\quad i=1,2,\ldots,n\, , \quad \alpha,\beta=1,2,\ldots,m\, ,\quad \gamma =1,2,\ldots,N\, ,$$ 
    and $M_{\alpha\beta}, K^i_{\alpha\beta}$ have the same definition as \eqref{eq:mkk}.
\end{remark}

\subsection{Stochastic advection by Lie transport}\label{sec:SALT}
We consider the example of EPDiff \cite{CHH2007, HM2005, HMR1998, HT2018a} to demonstrate the family of equations that exhibits the stochastic multisymplectic structure \eqref{eq:stoch multi symplectic}. In the deterministic case, EPDiff equation on an $n$-dimensional Riemannian manifold $\mcal{D}$ equipped with a metric $g$ can be derived from the following Clebsch variational principle
\begin{align}
    \mathcal{S} = \int_{t_0}^{t_1} \ell[u]\,dt + \int_{t_0}^{t_1}\int_\mcal{D} \pi_k\left(\p_t l_k + u_j\p_jl_{k}\right) \dd vol^s\,dt\,, \label{eq:det clebsch vp}
\end{align}
where $\ell[u]$ is an action defined on the space of vector fields $\mathfrak{X}(\mcal{D})$ (as well as its prolongations) and all variations are assumed to be arbitrary and vanishing at boundaries. The dependent variables in \eqref{eq:det clebsch vp} have the following interpretations that $u$ is the continuum Eulerian velocity, $l_k$ is the $k$-th fluid label and $\pi_k$ is the $k$-th  Lagrange multiplier enforcing the advection of $l_k$ by $u$. 
The critical point of the action \eqref{eq:det clebsch vp} gives the following relationships
\begin{align*}
    \delta u:\, \pi_k \nabla l_k = \frac{\delta \ell}{\delta u}\,,\quad \delta \pi_k:\, \p_t l_k + u_j\p_jl_{k} = 0\,,\quad \text{and}\quad  \delta l_k:\, \p_t \pi_k + \p_j(u_j\pi_k) = 0\,.
\end{align*}
Via a direct calculation \cite{CHH2007,Holm2011}, one obtains the Euler--Poincar\'e equation
\begin{align}
    \p_t \frac{\delta \ell}{\delta u} + \ad^*_{u} \frac{\delta \ell}{\delta u} = 0\,, \label{eq:det EP eq}
\end{align}
where the coadjoint operator $\ad^*: \mathfrak{X}(\mcal{D}) \times \mathfrak{X}^*(\mcal{D}) \rightarrow \mathfrak{X}^*(\mcal{D})$ is defined by 
\begin{align*}
    \scp{\ad^*_u m}{v} = \scp{m}{\ad_u v} = \scp{m}{-[u, v]}\,,\quad \forall\, u, v\in \mathfrak{X}(\mcal{D})\,, m \in \mathfrak{X}^*(\mcal{D})\,,
\end{align*}
with $\scp{\cdot}{\cdot}:\mathfrak{X}(\mcal{D})\times \mathfrak{X}^*(\mcal{D}) \rightarrow \mathbb{R}$  the (weak) duality pairing induced by the metric $g$, and $\delta$ denotes the variational derivative.  In the following, the metric $g$ is assumed to be Euclidean, so it is not necessary to distinguish between components with raised and lowered indices. 
In local coordinates, for $u = u_i\nabla_i \in \mathfrak{X}(\mcal{D})$ and $m = m_i \dd x^i\otimes \dd vol^s \in \mathfrak{X}^*(\mcal{D})$, the coadjoint operator has the coordinate expression 
\begin{align*}
    \left(\ad^*_u m\right)_i = u_j \nabla_j m_i + m_j \nabla_i u_j + m_i \nabla_j u_j\,.
\end{align*}
We consider a concrete choice of Lagrangian $\ell$, 
\begin{align}
    \ell[u] = \frac{1}{2}|| u ||^2_{H^1_\lambda} = \int_\mcal{D} \frac{1}{2}\left(|u|^2 + \lambda^2|\nabla u|^2\right)\,\dd vol^s\,,
\end{align}
where $H^1_\lambda$ is the $\lambda$-weighted $H^1$ norm for some $\lambda \in \mathbb{R}^+$. The resulting Euler--Poincar\'e equation becomes the EPDiff$(H^1_\lambda)$ equation
\begin{align}
    \p_t m_i + u_j \nabla_j m_i + m_j \nabla_i u_j + m_i \nabla_j u_j = 0\,,\qquad m_i = \left(1-\lambda^2\nabla^2\right)u_i\,. \label{eq:det EPdiff eq}
\end{align}
When the space is one-dimensional, this is equivalent to the Camassa--Holm equation originally derived in \cite{CH1993}. To cast the EPDiff$(H^1_\lambda)$ equation \eqref{eq:det EPdiff eq} to the multisymplectic formalism, we first define the base space variables as
\begin{align*}
    x^{0} = t\, , \quad \text{and} \quad   x^i\,, \quad i = 1,2,\ldots, n\, ,
\end{align*}
where $x=(x^1,x^2,\ldots,x^n)$ is the local coordinate on $\mcal{D}$. The dependent variable ${z}$ is defined as
\begin{align}
    z^i = u_i\,,\quad z^{n+j} = l_j\,,\quad z^{2n+j} = \pi_j\,, \quad z^{(i+2)n+j} = w_{ij}\,, \quad i,j = 1,2,\ldots,n \,,
\end{align}
Consider the following action \cite{CHH2007}
\begin{align}
    \mathcal{S} = \int_{t_0}^{t_1}\int_{\mcal{D}} \left(\frac{1}{2}u_i u_i - \frac{\lambda^2}{2}w_{ij}w_{ij} + \lambda^2 w_{ij} u_{i,j} + \pi_k(\p_t l_k + u_j l_{k,j})\right) \,\dd vol^s \,dt  \,,
\end{align}
where the auxiliary variables $w_{ij} = u_{i,j}$ are introduced such that the Lagrangian is affine in the space and time derivatives of the dependent variables $u_i, w_{ij}, \pi_i$ and $l_i$. Comparing with the abstract form of the Lagrangian \eqref{eq:det multi lag}, we have the 
following representation of $L^i_\alpha$:
\begin{equation}\label{eq:Lij}
\begin{aligned}
&L^0_i = 0\, , \quad L_{n+j}^0= \pi_j\, ,\quad L_{2n+j}^0=0\, ,\quad L_{(i+2)n+j}^0=0\, ,\\
&L^k_i = \lambda^2 w_{ik}\, , \quad L_{n+j}^k= \pi_j u_k\, ,\quad L_{2n+j}^k=0\, ,\quad L_{(i+2)n+j}^k=0\, ,
\end{aligned}
\end{equation}
for $i,j,k=1,2,\ldots,n$,
and the Hamiltonian $H$ is given by
\begin{align}\label{eq:HHH}
    H(z) = -\frac{1}{2}\left(u_iu_i - \lambda^2 w_{ij} w_{ij}\right) = -\frac{1}{2}\left(|u|^2 - \lambda^2 |w|^2\right)\,.
\end{align}
Taking the variations of $\mathcal{S}$ gives the following relationships
\begin{align*}
    \begin{split}
        \delta u_i &: u_i - \lambda^2 w_{ij,j} + \pi_k l_{k,i} = 0\, ,\\
        \delta \pi_k &: \p_t l_k + u_j l_{k,j} = 0\, ,\\
        \delta l_k &: -\p_t \pi_k - (\pi_k u_j)_{,j} = 0\, ,\\
        \delta w_{ij} &: w_{ij} = u_{i,j}\,,
    \end{split}
\end{align*}
which can be assembled into the multisymplectic form below
\begin{align}
\begin{split}
    &\begin{pmatrix}
        0 & 0 &0 &0 \\
        0 &0 & -\p_t & 0\\
        0& \p_t & 0 & 0 \\
        0 & 0 & 0 & 0
    \end{pmatrix}
    \begin{pmatrix}
    u_i \\l_k \\ \pi_k\\ w_{ij}
    \end{pmatrix} +
    \begin{pmatrix}
        0 & \pi_k\partial_i & 0 & -\lambda^2 \partial_j \\
        -\pi_k\partial_i & 0 & -u_j\partial_j & 0\\
        0 & u_j\partial_j & 0 & 0\\
        \lambda^2\partial_j & 0 & 0 & 0
    \end{pmatrix}
    \begin{pmatrix}
    u_i \\l_k \\ \pi_k\\ w_{ij}
    \end{pmatrix}
    = \nabla_z H \,.
\end{split}
\end{align}
The $1$-form conservation law in the dependent variable basis \eqref{eq:det one form conservations laws} for the EPDiff$(H^1_\lambda)$ case can be calculated to be
\begin{align*}
    \p_t\left(\pi_k \dd l_k\right) + \p_j\left(\lambda^2w_{ij} \dd u_{i,j} + \pi_ku_j \dd l_k\right) = \dd L \,,
\end{align*}
and the structural conservation law corresponding to \eqref{eq:det structural conservation laws} is expressed as
\begin{align*}
    \p_t \left(\dd\pi_k \wedge \dd l_k\right) + \p_j\left(\lambda^2\dd w_{ij}\wedge \dd u_i + u_j\dd \pi_k\wedge \dd l_k + \pi_k\dd u_j \wedge \dd l_k\right) = 0\,.
\end{align*}

The stochastic perturbation under consideration is in the form of the stochastic advection by Lie transport (SALT) framework \cite{Holm2015}. For an unspecified Lagrangian, the SALT equations can be derived using a stochastic perturbation to the Clebsch variational principle \eqref{eq:det clebsch vp} for
\begin{align}
    \mathcal{S} = \int_{t_0}^{t_1} \ell[u]\,dt +  \int_\mcal{D} \int_{t_0}^{t_1} \pi_k\left(\dd l_k + \mcal{L}_u l_{k}\,dt + \mcal{L}_{\xi} l_{k} \circ dW_t \right)\, \dd vol^s \,, \label{eq:stoch clebsch vp}
\end{align}
where $\xi = \xi(x) \in \mathfrak{X}(\mcal{D})$ is a prescribed vector field with possible dependence of the domain. For simplicity, we consider the case where $\xi =(\xi^1,\xi^2,\ldots,\xi^n)\in \mathbb{R}^n$ are constants such that the stochastic action conforms to the case considered in \eqref{eq:stoch ms action}. The resulting stochastic Euler--Poincar\'e equation is give by
\begin{align}
     d \frac{\delta \ell}{\delta u} + \left(\ad^*_{u} \frac{\delta \ell}{\delta u}\right) \,dt + \left(\ad^*_{\xi} \frac{\delta \ell}{\delta u} \right) \circ dW_t = 0\,. \label{eq:stoch EP eq}
\end{align}
Analytically, the type of noise appearing in \eqref{eq:stoch EP eq} are known as transport noise. When $\xi$ is constant, the addition of noise corresponds to a translation of space
\begin{align}
    x \, \mapsto \wt{x}(x,t)\, := x - \xi W(t)\,, \quad \text{for} \quad W(t) = W_0 - \int_0^t dW_s\,. \label{eq:stoch x translation}
\end{align}
Thus, for a solution $u(x,t)$ to the deterministic equation \eqref{eq:det EP eq}, $u(\wt{x}(x,t), t)$ is a solution to the stochastic equation \eqref{eq:stoch EP eq}. In the case when $\ell$ corresponds to the Lagrangian for fluid equations such as Euler's fluid, analytical properties have been established in, e.g., \cite{CFH19, CHLN2022109632} for non-constant $\xi$.
Choosing the EPDiff$(H^1_\lambda)$ Lagrangian, we have the SALT EPDiff equation
\begin{align}
    d m_i + \left(u_j \nabla_j m_i + m_j \nabla_i u_j + m_i \nabla_j u_j\right)\,dt + \left(\xi^j \nabla_j m_i + m_j \nabla_i \xi^j + m_i \nabla_j \xi^j\right)\circ dW_t = 0\,, \label{eq:stoch EPdiff eq}
\end{align}

Noticing that the inclusion of noise only introduces an affine perturbation in the spatial derivative of the dependent variables in the action, \eqref{eq:stoch clebsch vp}, the stochastic multisymplectic variational principle is given by
\begin{align}
     \mathcal{S} = \int \left[\int_{t_0}^{t_1} \left(\frac{1}{2}u_i u_i - \frac{\lambda^2}{2}w_{ij}w_{ij} + \lambda^2 w_{ij}u_{i,j}\right)\,dt + \pi_k(d l_k + u_j l_{k,j}\,dt + \xi^j l_{k,j}\circ dW_t)\right]\dd vol^s\,.
\end{align}
Comparing with the abstract form of the stochastic multisymplectic action \eqref{eq:stoch ms action}, we have \eqref{eq:Lij} and
\begin{align*}
\wt{L}^k_i=0\, ,\quad \wt{L}^k_{n+j}=\pi_j\xi_k\, ,\quad \wt{L}^k_{2n+j} =0\, , \quad \wt{L}^k_{(i+2)n+j}=0\, ,\quad i,j,k=1,2,\ldots,n \, .
\end{align*}
Taking variations will give the stochastic Euler--Lagrange equations
\begin{align}
    \begin{split}
        \delta u_i &: u_i - \lambda^2 w_{ij,j} + \pi_k l_{k,i} = 0\,,\\
        \delta \pi_k &: d l_k + u_j l_{k,j}\,dt + \xi^j l_{k,j}\circ dW_t = 0\,,\\
        \delta l_k &: -d \pi_k - (\pi_k u_j)_{,j}\,dt - (\pi_k \xi^j)_{,j}\circ dW_t = 0\,,\\
        \delta w_{ij} &: w_{ij} = u_{i,j}\,,
    \end{split}
\end{align}
which can be assembled to the following multisymplectic form below (for constant $\xi$)
\begin{align}
\begin{split}
    \begin{pmatrix}
        0 & 0 &0 &0 \\0 &0 & -d & 0\\ 0& d &0&0\\0&0&0&0
    \end{pmatrix}
    \begin{pmatrix}
    u_i \\l_k \\ \pi_k\\ w_{ij}
    \end{pmatrix} &+
    \begin{pmatrix}
        0 & \pi_k\partial_i & 0 & -\lambda^2 \partial_j \\
        -\pi_k\partial_i & 0 & -u_j\partial_j & 0\\
        0 & u_j\partial_j & 0 & 0\\
        \lambda^2\partial_j & 0 & 0 & 0
    \end{pmatrix}
    \begin{pmatrix}
    u_i \\l_k \\ \pi_k\\ w_{ij}
    \end{pmatrix}\,dt \\
    &  +
    \begin{pmatrix}
        0 &0 &0 &0 \\
        0 &0 & -\xi^j\partial_j & 0\\
        0& \xi^j\partial_j &0&0\\
        0&0&0&0
    \end{pmatrix}
    \begin{pmatrix}
    u_i \\l_k \\ \pi_k\\ w_{ij}
    \end{pmatrix}\circ dW_t  
    = \nabla_z H\,dt\,,
\end{split}
\end{align}
where $H(z)$ is defined as in the deterministic case, Eq. \eqref{eq:HHH}. 
\begin{remark}
    In the particular case where $\wt{L}^k({z}) = \xi^k L^0({z})$ for constant $\xi^k \in \mathbb{R}$, the resulting stochastic multisymplectic equation will have the spatial translation property defined in equation \eqref{eq:stoch x translation}. 
\end{remark}

Comparing with the general form of \eqref{eq:stochastic 1-form conservation law}, we have for stochastic EPDiff$(H^1_{\lambda})$ the 1-form conservation law:
\begin{align}
\begin{split}
     d(\pi_k \dd l_k) &+ (\lambda^2 w_{ij} \dd u_i \,dt + \pi_k u_j \dd l_k \,dt + \pi_k \xi^j \dd l_k \circ dW_t)_{,j} \\
    &  = \dd\left(\left(\frac{1}{2}u_i u_i - \frac{\lambda^2}{2}w_{ij}w_{ij} + \lambda^2 w_{ij}u_{i,j}\right)\,dt + \pi_k(d l_k + u_j l_{k,j}\,dt + \xi^j l_{k,j}\circ dW_t)\right)\, ,
\end{split}\label{eq:stoch epdiff conservation law}
\end{align}
and the stochastic structural conservation law
\begin{align*}
   d\left(\dd\pi_k \wedge \dd l_k\right) + \p_j\left(\lambda^2\dd w_{ij}\wedge \dd u_i + u_j\dd \pi_k\wedge \dd l_k + \pi_k\dd u_j \wedge \dd l_k\right)\,dt + \p_j\left(\xi^j\dd \pi_k\wedge \dd l_k \right)\circ dW_t = 0\,.
\end{align*}
Additionally, one can show that the pullback to the spatial base variables of equation \eqref{eq:stoch epdiff conservation law} gives the stochastic EPDiff equation \eqref{eq:stoch EP eq}. 


\subsection{Stochastic nonlinear Schr\"odinger equations}\label{subsec:NLS}
One classical  example of multisymplectic PDEs is the NLS equation (e.g., \cite{Bridges1997n,chen2002symplectic}). On an $n$-dimensional Riemannian manifold $\mathcal{D}$ equipped with metric $g$ which is assumed to be the Euclidean metric for simplicty, the NLS equation can be written in complex wave function form as
\begin{align}
    i\p_t \psi = -\Delta \psi + 2\kappa|\psi|^2\psi\,.
\end{align}
Here, $\psi$ is the complex wave function and $\kappa \in \mathbb{R}$ is a real constant. To write the NLS equation in multisymplectic form, we write the NLS equation in the real and imaginary components of $\psi = p + iq$ to have the system of equations
\begin{equation}
    \begin{aligned}
    \p_t p &= -\Delta q + 2\kappa(p^2 + q^2)q\, ,\\
     \p_t q &= \Delta p - 2\kappa(p^2 + q^2)p\,. \label{eq:det NLS}
\end{aligned}
\end{equation}
Introducing the auxiliary variables $v_j$ and $w_j$ defined by $v_j = \p_j p$ and $w_j = \p_j q$, respectively, we can in turn define the vector of dependent variables ${z}$ as
\begin{align*}
    z^1 = p\,,\quad z^{2} = q\,,\quad z^{2 + j} = v_{j}\,, \quad z^{2+n+j} = w_{j}\,, \quad j = 1,2,\ldots,n \,,
\end{align*}
and consider the following action
\begin{align}
    \begin{split}
         \mathcal{S} = \int_{t_0}^{t_1}\int_{\mcal{D}}\left(p\p_t q +v_j\p_jp + w_j \p_j q - H(z)\right)\,\dd vol^s\,dt\,,
    \end{split}
\end{align}
where the Hamiltonian $H$ is given by
\begin{align}
    H({z}) = -\frac{1}{2}\left(\kappa(p^2 + q^2)^2 - |v|^2 - |w|^2\right)\,.
\end{align}
Compare with the abstract form of the Lagrangian \eqref{eq:det multi lag}, we have the following representation of $L^j_\alpha$:
\begin{align}
\begin{split}
    &L^0_1 = 0\,,\quad L^0_2 = p\,, \quad L^0_{2+j} = 0\,,\quad L^0_{2+n+j} = 0\,,\\
    &L^k_1 = v_k\,,\quad L^k_2 = w_k\,, \quad L^k_{2+j} = 0\,,\quad L^k_{2+n+j} = 0\,,
\end{split}
\end{align}
for $j,k = 1,2,\ldots, n$.
The NLS equations can be assembled into the multisymplectic form \eqref{eq:stoch multi symplectic} as follows
\begin{align}
    \begin{pmatrix}
        0 & \p_t & 0 & 0\\
        -\p_t & 0 & 0 & 0\\
        0 & 0 & 0 & 0\\
        0 & 0 & 0 & 0\\
    \end{pmatrix}
    \begin{pmatrix}
        p\\ q\\v_j \\ w_j
    \end{pmatrix}
    + 
    \begin{pmatrix}
        0 & 0& -\p_j & 0\\
        0 & 0 & 0 & -\p_j\\
        \p_j & 0 & 0 & 0\\
        0 & \p_j & 0 & 0\\
    \end{pmatrix}
    \begin{pmatrix}
        p\\ q\\v_j \\ w_j
    \end{pmatrix}
    = 
    \begin{pmatrix}
       -2\kappa(p^2+ q^2)p \\ -2\kappa(p^2+ q^2)q \\ v_j \\ w_j
    \end{pmatrix}
    = \nabla_{{z}} H({z}) \,.\label{eq:det ms NLS}
\end{align}

\begin{remark}[Notation on the expression of multisymplectic matrices]
To write down explicitly the multisymplectic matrices $M$ and $K^j$ in the form of equation \eqref{eq:stoch multi symplectic} for the NLS equations, we introduce the following notations. Define 
\begin{align*}
    e_j \in \mathbb{R}^{n} = (0,\ldots, 0, 1,0,\ldots, 0)\,,
\end{align*}
where the $1$ entry is in the $j$-th index. Additionally define $0_n \in \mathbb{R}^{1\times n}$ and $0_{n\times n} \in \mathbb{R}^{n\times n}$ as vectors and matrices of $0$s of shape $1\times n$ and $n \times n$, respectively. We can write $M$ and $K^j$ as 
\begin{align}
    M = 
    \begin{pmatrix}
        0 & 1  & 0_n & 0_n\\
        -1 & 0  & 0_n & 0_n\vspace{0.1cm}\\
        0_n^T & 0_n^T  & 0_{n\times n} & 0_{n\times n} \vspace{0.15cm}\\
        0_n^T & 0_n^T  & 0_{n\times n} & 0_{n\times n}
    \end{pmatrix}\quad 
    \text{and}\quad 
    K^j = 
    \begin{pmatrix}
        0 & 0 & -e_j & 0_n\\
        0 & 0  & 0_n & -e_j\vspace{0.1cm}\\
        e_j^T & 0_n^T  & 0_{n\times n} & 0_{n\times n}\vspace{0.15cm}\\
        0_n^T & e_j^T & 0_{n\times n} & 0_{n\times n}
    \end{pmatrix}\,.
\end{align}    
However, we believe that the implicit summation notation as used in equation \eqref{eq:det ms NLS} is more convenient and we shall persist with the notation for the rest of the section. 
\end{remark}
The $1$-form conservation law in the dependent variable basis \eqref{eq:det one form conservations laws} for the NLS equations can be found as
\begin{align*}
    \p_t\left(p \dd q\right) + \p_j\left(v_j\dd p + w_j\dd q\right) = \dd \left(p\p_t q +v_j\p_jp + w_j \p_j q + \frac{1}{2}\left(  \kappa\left(p^2 + q^2\right)^2- |v|^2 -|w|^2\right)\right) \,,
\end{align*}
and the structural conservation law corresponding to \eqref{eq:det structural conservation laws} is expressed as
\begin{align*}
    \p_t\left(\dd p \wedge \dd q\right) + \p_j\left(\dd v_j\wedge \dd p +\dd w_j\wedge \dd q\right) =  0\,.
\end{align*}
\paragraph{Stochastic transport.}
Following the stochastic EPDiff example presented in Section \ref{sec:SALT}, we additionally consider the equivalent transport noise applied to the NLS equations. In this case, we consider the stochastic perturbation defined by taking $\wt{L}^j_\alpha = \xi^j L^0_\alpha$ where $\{\xi^j\}_{j = 1,2,\ldots, n}$ are real coefficients without introducing the stochastic Hamiltonian $\wt{H}$. Thus, we have $\wt{K}^j = \xi^j M$ and we obtain the following stochastic NLS equations in the multisymplectic form
\begin{align}
    \begin{split}
            \begin{pmatrix}
        0 & d & 0 & 0\\
        -d & 0 & 0 & 0\\
        0 & 0 & 0 & 0\\
        0 & 0 & 0 & 0\\
    \end{pmatrix}
    \begin{pmatrix}
        p\\ q\\v_j \\ w_j
    \end{pmatrix}
   & + 
    \begin{pmatrix}
        0 & 0& -\p_j & 0\\
        0 & 0 & 0 & -\p_j\\
        \p_j & 0 & 0 & 0\\
        0 & \p_j & 0 & 0\\
    \end{pmatrix}
    \begin{pmatrix}
        p\\ q\\v_j \\ w_j
    \end{pmatrix}
    \,dt \\
    & + 
    \begin{pmatrix}
        0 & \xi^j\p_j & 0 & 0\\
        -\xi^j\p_j & 0 & 0 & 0\\
        0 & 0 & 0 & 0\\
        0 & 0 & 0 & 0\\
    \end{pmatrix}
    \begin{pmatrix}
        p\\ q\\v_j \\ w_j
    \end{pmatrix}
    \circ dW_t
    = 
    \begin{pmatrix}
       -2\kappa(p^2+ q^2)p \\ -2\kappa(p^2+ q^2)q \\ v_j \\ w_j
    \end{pmatrix}
    \,dt\,.
    \end{split}\label{eq:SALT multi NLS}
\end{align}
Expressing in terms of $p, q$ variables we have 
\begin{align}
    \begin{split}
        &dp + \xi^j\p_jp \circ dW_t = -\Delta q\,dt + 2\kappa(p^2 + q^2)q\,dt  \,,\\
        &dq + \xi^j\p_jq \circ dW_t = \Delta p\,dt - 2\kappa(p^2 + q^2)p\,dt \,. 
    \end{split}\label{eq:SALT NLS}
\end{align}
Since $\xi^j \in \mathbb{R}$ are constants, we again have the property that if $q(x,t), p(x,t)$ are solutions to the deterministic NLS equations \eqref{eq:det NLS}, then $q\left(x - \xi W(t), t\right), p\left(x - \xi W(t), t\right)$ are solutions to the stochastic NLS equations \eqref{eq:SALT NLS}, where $\xi = (\xi^1, \xi^2,\ldots,\xi^n)^T \in \mathbb{R}^n$. 

Substituting the expression of $\wt{L}^j_\alpha$ into the general form of the stochastic conservation laws, \eqref{eq:stochastic 1-form conservation law} and \eqref{eq:stoch structural conservation laws}, we have
\begin{align*}
d\left(p \dd q\right) &+ \p_j\left(\left(v_j\dd p + w_j\dd q\right)\,dt + \xi^j p\dd q\circ dW_t\right)\\ 
& = \dd \left(p dq + \left(v_j\p_jp + w_j \p_j q\right)\,dt + \xi^j p\p_j q \circ dW_t + \frac{1}{2}\left(  \kappa\left(p^2 + q^2\right)^2 -  |v|^2 - |w|^2\right)\,dt\right) \,,
\end{align*}
and
\begin{align*}
    d\left(\dd p \wedge \dd q\right) + \p_j\left(\left(\dd v_j\wedge \dd p +\dd w_j\wedge \dd q\right)dt + \xi^j \dd p \wedge \dd q \circ dW_t\right) =  0\,,
\end{align*}
respectively, associated to the stochastic NLS equations \eqref{eq:SALT NLS}.

\paragraph{Stochastic dispersion.}
Motivated from \cite{CHLZ2017}, let us consider the stochastic perturbation defined by taking $\wt{K}^j = \e K^j$ and stochastic Hamiltonian $\wt{H} = \frac{\e}{2}\left(|v|^2 + |w|^2\right)$. In this case, we have $\wt{L}^j_\alpha := \e L^j_\alpha$ and the stochastic multisymplectic NLS equations can be written as 
\begin{align}
\begin{split}
    &\begin{pmatrix}
        0 & d & 0 & 0\\
        -d & 0 & 0 & 0\\
        0 & 0 & 0 & 0\\
        0 & 0 & 0 & 0\\
    \end{pmatrix}
    \begin{pmatrix}
        p\\ q\\v_j \\ w_j
    \end{pmatrix}
    + 
    \begin{pmatrix}
        0 & 0& -\p_j & 0\\
        0 & 0 & 0 & -\p_j\\
        \p_j & 0 & 0 & 0\\
        0 & \p_j & 0 & 0\\
    \end{pmatrix}
    \begin{pmatrix}
        p\\ q\\v_j \\ w_j
    \end{pmatrix}
    \,dt \\
    &\qquad \qquad + 
    \begin{pmatrix}
        0 & 0& -\e\p_j & 0\\
        0 & 0 & 0 & -\e\p_j\\
        \e\p_j & 0 & 0 & 0\\
        0 & \e\p_j & 0 & 0\\
    \end{pmatrix}
    \begin{pmatrix}
        p\\ q\\v_j \\ w_j
    \end{pmatrix}
    \circ dW_t
    = 
    \begin{pmatrix}
       -2\kappa(p^2+ q^2)p \\ -2\kappa(p^2+ q^2)q \\ v_j \\ w_j
    \end{pmatrix}
    \,dt
    + \e
    \begin{pmatrix}
        0 \\ 0 \\ v_j \\ w_j
    \end{pmatrix}
    \circ dW_t \,.
\end{split}\label{eq:stoch multi NLS}
\end{align}
Expressing the stochastic NLS equations can be written in terms of $p, q$ variables only, we have
\begin{align}
    \begin{split}
        dp &= -\Delta q\,dt + 2\kappa(p^2 + q^2)q\,dt - \e\Delta q \circ dW_t \,,\\
        dq &= \Delta p\,dt - 2\kappa(p^2 + q^2)p\,dt + \e\Delta p\circ dW_t\,. 
    \end{split}\label{eq:stoch NLS}
\end{align}
\begin{remark}[Consistency of auxiliary variables $v$ and $w$]
    Recall that in the deterministic case, the auxiliary variables $v_j$ and $w_j$ are defined by $v_j = \p_j p$ and $w_j = \p_j q$. In the stochastic case, these definitions are again enforced in the third and fourth components of the multisymplectic equation \eqref{eq:stoch multi NLS} which read
    \begin{align*}
        \p_j p\,dt + \e\p_j p\circ dW_t &= v_j\,dt + \e v_j \circ dW_t\,,\\
        \p_j q\,dt + \e\p_j q\circ dW_t &= w_j\,dt + \e w_j \circ dW_t\,.
    \end{align*}
    That is, the conditions $v_j = \p_j p$ and $w_j = \p_j q$ are constrained by both the $dt$ and $dW_t$ components of the motion and the stochastic Hamiltonian $\wt{H}$ is essential in enforcing these constraints.
\end{remark}
The stochastic local conservation laws that are associated with the stochastic NLS equations \eqref{eq:stoch NLS} are the $1$-form quasi-conservation law corresponding to \eqref{eq:stochastic 1-form conservation law}, expressed as
\begin{align*}
    &d\left(p \dd q\right) + \p_j\left(v_j\dd p + w_j\dd q\right)\left(dt + \e\circ dW_t\right) \\
    &\qquad = \dd \left(p dq + \left(v_i\p_ip + w_i \p_i q\right)\left(dt + \e\circ dW_t\right)  - \frac{1}{2}\left(\left(|v|^2 + |w|^2\right)(dt + \e\circ dW_t) - \kappa\left(p^2 + q^2\right)^2\,dt\right) \right) \,,
\end{align*}
and the structural conservation law corresponding to \eqref{eq:stoch structural conservation laws}, expressed as
\begin{align*}
    d\left(\dd p \wedge \dd q\right) + \p_j\left(\dd v_j\wedge \dd p +\dd w_j\wedge \dd q\right)\left(dt + \e\circ dW_t\right) =  0\,.
\end{align*}
Both stochastic NLS equations, \eqref{eq:SALT NLS} and \eqref{eq:stoch NLS}, preserve the properties of global conservation of wave function density $|\psi|^2$ and linear momentum $J(\psi, \psi^*) = \psi^*\nabla\psi = p\nabla q - q\nabla p$ from the deterministic NLS equations. That is,
\begin{align*}
    d \int_{\mcal{D}} |\psi|^2\, \dd vol^s = 0\,,\quad \textrm{and} \quad d \int_{\mcal{D}} \psi^*\nabla\psi\, \dd vol^s = 0\,,
\end{align*} 
which can be verified through direct calculation.

\section{Stochastic multisymplectic methods}\label{sec:stochastic multisymplectic discrete}
The application of collocation methods to multisymplectic PDEs of the form \eqref{eq:det multi symplectic} was first introduced in \cite{Reich2000} and later generalised in \cite{HLS2005}. When applied to multisymplectic PDEs, collocation methods discretise the system through Runge--Kutta methods in both temporal and spatial dimensions. This will produce a discrete version of local conservation of multisymplectic form. Additionally, it was shown that through appropriate choices of the coefficients appearing in the Runge--Kutta methods, quadratic invariants can be preserved exactly. Thus, collocation methods can preserve energy and momentum exactly for linear multisymplectic PDEs.

In extension to the stochastic multisymplectic PDEs of the form \eqref{eq:stoch multi symplectic}, one can discretise system \eqref{eq:stoch multi symplectic} using a deterministic Runge--Kutta method with $s$ stages for the spacial direction and a stochastic Runge--Kutta method with $r$ stages for the temporal direction. In this section, we will adhere to the strict Einstein summation convention, where repeated superscripts and subscripts imply summation.

\subsection{Stochastic collocation methods}
For simplicity, we consider the stochastic collocation methods for the stochastic multisymplectic PDEs \eqref{eq:stoch multi symplectic}, where the multisympletic structure matrices $\mathbf{M}=(M_{\alpha\beta}), \mathbf{K}^i=(K^i_{\alpha\beta})$ and $ \wt{\mathbf{K}}^i=(\wt{K}^i_{\alpha\beta})$ are constant. Restricting to one spatial dimensions, the general form of the stochastic PDEs is given by
\begin{align}
    M_{\alpha\beta
    }d z^\beta + K_{\alpha\beta}\p_x z^\beta \,dt + \wt{K}_{\alpha\beta}\p_x z^\beta \circ dW_t = \frac{\p H}{\p z^\alpha}\,dt + \frac{\p \wt{H}}{\p z^\alpha} \circ dW_t\,, \label{eq:stochastic ms 1d}
\end{align}
where $\mathbf{M}$, $\mathbf{K}$ and $\wt{\mathbf{K}}$ are all constant matrices.

The collocation points in the $(t, x)$ domain are determined as in the standard setting. We consider a discretisation of the integration domain through equidistant points $  [x_0,x_1, \ldots,x_M] \times [t_0, t_1,\ldots,t_N]$ such that $x_{n+1} - x_{n} = \triangle x$  and $t_{k+1} - t_{k} = \triangle t$  for all $n \in \{0,1,\ldots, M-1\}$ and  $k \in \{0,1,\ldots, N-1\}$. Between every spatial and temporal  discretisation points, we consider $s$ and $r$ collocation points given by $x_n + c_m \triangle x$ and $t_k + d_i\triangle t $, respectively, where $m \in \{1,2,\ldots, s\}$ and $i \in \{1,2,\ldots, r\}$. We introduce the following notations as the evaluation of $z$ at the collocation points, 
\begin{align*}
    &z_{0, m} := z(t_k, x_n + c_m\triangle x)\,,\quad z_{i, 0} := z(t_k + d_i \triangle t, x_n)\,,\quad Z_{i,m} := z(t_k + d_i\triangle t, x_n + c_m\triangle x)\,,\\
    &z_{1, m} := z(t_{k+1}, x_n + c_m\triangle x)\,,\quad \textrm{and} \quad z_{i,1} := z(t_k + d_i\triangle t , x_{n+1})\,.
\end{align*}
For the approximation of integration against Brownian motion, we  use the first-order approximation \cite{BB2001, KP1992},
\begin{align}
    \triangle W = W(t_{k+1}) - W(t_k) \approx \int_{t_k}^{t_{k+1}} \circ dW_t \,,\label{eq:brownian increments}
\end{align}
which is the increment of the Wiener process over the time discretisation and is a Gaussian random variable with zero mean and variance $\triangle t$. As we are only using the first-order approximation, one can only expect local mean square convergence of order at most $1.0$ \cite{BB2001}. For higher-order methods, one needs to consider higher-order approximations such as $J = \int_{t_k}^{t_{k+1}}\int_{t_k}^s\circ dW_{s}\,ds_1$. 

Using above quantities, we follow \cite{ZJ2019} and consider the stochastic Runge--Kutta methods with $s$-stage in spatial direction and $r$-stage in temporal direction  defined by
\begin{align}
    \begin{split}
        &Z_{i,m} = z_{i, 0} + \triangle x\sum^s_{n=1}a_{mn}\left(\delta_xZ_{i,n}\right)\,,\\
        &z_{i,1} = z_{i, 0} + \triangle x\sum^s_{n=1}b_{n}\left(\delta_xZ_{i,n}\right)\,,\\
        &Z_{i,m} = z_{0,m} + \triangle t\sum^r_{j=1}\ob{a}_{ij}\left(\delta_t^A Z_{j,m}\right) + \triangle W\sum^r_{j=1}\ob{\alpha}_{ij}\left(\delta_t^M Z_{j,m}\right)\,,\\
        &z_{1,m} = z_{0,m} + \triangle t\sum^r_{j=1}\ob{b}_{j}\left(\delta_t^A Z_{j,m} \right)+ \triangle W\sum^r_{j=1}\ob{\beta}_{m}\left(\delta_t^M Z_{j,m}\right)\,,
         \label{eq:stochastic RK}
    \end{split}
\end{align}
together with the defining relations
\begin{align}
    \mathbf{M}\left( \delta_t^A Z_{i,m} \right) + \mathbf{K} \left( \delta_x Z_{i,m} \right)= \nabla_z H(Z_{i,m})\,,\quad \mathbf{M} \left( \delta_t^M Z_{i,m} \right)+ \wt{\mathbf{K}} \left( \delta_x Z_{i,m} \right)= \nabla_z\wt{H}(Z_{i,m})\,. \label{eq:stochastic RK Ham}
\end{align}
 Here, we have used the notations 
$$\nabla_z H(Z_{i,m}) = \nabla_z H(z) \big|_{z= Z_{i,m}}\, , \quad \text{and} \quad  \nabla_z\wt{H}(Z_{i,m}) =\nabla_z \wt{H}(z) \big|_{z= Z_{i,m}} \, .$$
In \eqref{eq:stochastic RK}, the Runge--Kutta coefficients $a_{ij}, \ob{a}_{ij}$, $\alpha_{ij}$, $b_i$, $\ob{b}_i$ and $\ob{\beta}_i$ are defined to satisfy the following consistency conditions
\begin{align*}
    d_i = \sum^r_{j=1}\ob{a}_{ij}\, , \quad d_i = \sum^r_{j=1}\ob{\alpha}_{ij}\,, \quad c_m = \sum^s_{n=1}a_{mn}\, , \quad \textrm{and} \quad \sum_{j=1}^r \ob{b}_j = \sum_{j=1}^r \ob{\beta}_j = \sum_{n=1}^s b_n = 1\,.
\end{align*}
The discrete operators $\delta_x$, $\delta^A_t$ and $\delta^M_t$ in  \eqref{eq:stochastic RK} and \eqref{eq:stochastic RK Ham} are defined as the discrete approximations to the spatial and temporal derivatives, respectively. In this work, we will use the standard forward differences,
\begin{align*}
    \delta_x Z_{i,n} = \frac{Z_{i, n+1} - Z_{i, n}}{\triangle x}\,,\quad \delta_t^A Z_{i,n} = \frac{Z_{i+1, n} - Z_{i, n}}{\triangle t}\,,\quad \delta_t^M Z_{i,n} = \frac{Z_{i+1, n} - Z_{i, n}}{\triangle W}\,.
\end{align*}
\begin{theorem}\label{thm: multi symplectic form thm}
When the following symplecticity conditions are satisfied
\begin{align}
    b_m b_n - a_{mn}b_m - b_n a_{nm} = 0\, , \quad \forall m,n=1,2,\ldots, s \label{eq:stochastic multisymplectic cond 1}
\end{align}
and 
\begin{align}
    \begin{split}
        \ob{b}_i\ob{b}_j - \ob{a}_{ji}\ob{b}_j - \ob{b}_i\ob{a}_{ij} &= 0\,,\\
        \ob{\beta}_i\ob{\beta}_j - \ob{\alpha}_{ji}\ob{\beta}_j - \ob{\beta}_i\ob{\alpha}_{ij} &= 0\,,\\
        \ob{b}_i\ob{\beta}_j - \ob{\alpha}_{ji}\ob{b}_j - \ob{b}_i\ob{\alpha}_{ij} &= 0\,,\\
        \ob{\beta}_i\ob{b}_j - \ob{a}_{ji}\ob{\beta}_j - \ob{\beta}_i\ob{a}_{ij} &= 0\,, \quad \forall i,j = 1,2,\ldots,r\,, \label{eq:stochastic multisymplectic cond 2}
    \end{split}
\end{align}
the Runge--Kutta methods defined above are multisymplectic, and we have the following discrete multisymplectic form conservation law,
\begin{align}
     \sum^s_{m=1}b_m\left[\omega_{1,m} - \omega_{0,m}\right]\triangle x + \sum^r_{i=1}\ob{b}_i\left[\kappa_{i, 1} - \kappa_{i, 0} \right]\triangle t + \sum^r_{i=1}\ob{\beta}_i\left[\wt{\kappa}_{i, 1} - \wt{\kappa}_{i, 0} \right]\triangle W = 0\,, \label{eq:discrete stoch ms CL}
\end{align}
where the discrete multisymplectic forms are defined by
\begin{align*}
    \omega_{i, m} := \frac12 \dd z_{i,m}\wedge\mathbf{M}\dd z_{i,m}\,,\quad \kappa_{i,m} := \frac12 \dd z_{i,m}\wedge\mathbf{K}\dd z_{i,m}\,,\quad \textrm{and} \quad \wt{\kappa}_{i,m} :=\frac12 \dd z_{i,m}\wedge\wt{\mathbf{K}} \dd z_{i,m}\,.
\end{align*}
\end{theorem}
\begin{remark}
    Before continuing to the proof of the theorem, we remark that the symplecticity conditions \eqref{eq:stochastic multisymplectic cond 2} are precisely the symplecticity conditions of the stochastic symplectic partitioned Runge--Kutta methods for SDEs \cite{HT2018b}. 
\end{remark}
\begin{proof}
The proof is by direct calculation. First we compute discrete evolution of the $1$-form $\dd z$ by taking the exterior derivative of the stochastic collocation method \eqref{eq:stochastic RK} to have
\begin{align}
    \begin{split}
        &\dd Z_{i,m} = \dd z_{i, 0} + \triangle x\sum^s_{n=1}a_{mn}\dd\left(\delta_xZ_{i,n}\right)\,,\\
        &\dd z_{i,1} = \dd z_{i, 0} + \triangle x\sum^s_{n=1}b_{n}\dd\left(\delta_xZ_{i,n}\right)\,,\\
        &\dd Z_{i,m} = \dd z_{0,m} + \triangle t\sum^r_{j=1}\ob{a}_{ij}\dd\left(\delta_t^A Z_{j,m}\right) + \triangle W\sum^r_{j=1}\ob{\alpha}_{ij}\dd\left(\delta_t^M Z_{j,m}\right)\,,\\
        &\dd z_{1,m} = \dd z_{0,m} + \triangle t\sum^r_{j=1}\ob{b}_{j}\dd\left(\delta_t^A Z_{j,m}\right) + \triangle W\sum^r_{j=1}\ob{\beta}_{m}\dd\left(\delta_t^M Z_{j,m}\right)\,,
    \end{split} \label{eq:stochastic RK 1-form}
\end{align}
where the notations for the evaluation of $1$-forms at discrete points follow from the notations of evaluation of functions. Then, the expression of $\kappa_{i,1}$ can be expanded as
\begin{align}
    \begin{split}
    \kappa_{i,1} &= \frac12\left(\dd z_{i, 0} + \triangle x\sum^s_{n=1}b_{n} \dd \left(\delta_xZ_{i,n}\right)\right) \wedge \mathbf{K}\left(\dd z_{i,0} + \triangle x\sum^s_{n=1}b_{n} \dd \left(\delta_x Z_{i,n}\right)\right)\\
    & = \frac12 \dd z_{i, 0}\wedge \mathbf{K} \dd z_{i, 0} +\frac12 \triangle x\sum^s_{n=1}b_n \left(\dd z_{i, 0} \wedge \mathbf{K} \dd \left(\delta_xZ_{i,n}\right) + \dd \left(\delta_xZ_{i,n}\right)\wedge \mathbf{K}\dd z_{i, 0} \right)\\
    & \qquad \qquad \qquad + \frac12(\triangle x)^2\sum^s_{m=1}\sum^s_{n=1}b_mb_n  \dd \left(\delta_xZ_{i,n}\right) \wedge \mathbf{K} \dd \left(\delta_x Z_{i,m}\right)\\
    & =  \frac12 \dd z_{i, 0}\wedge \mathbf{K} \dd z_{i, 0} + \frac12\triangle x\sum^s_{n=1}b_n \dd Z_{i,n}\wedge \mathbf{K}  \dd \left(\delta_x Z_{i,n}\right) + \frac12\triangle x\sum^s_{n=1}b_n \dd\left(\delta_x Z_{i,n}\right) \wedge \mathbf{K}\dd Z_{i,n}  \\
    & \qquad \qquad \qquad +\frac12 (\triangle x)^2\sum^s_{n=1}\sum^s_{m=1}\left(b_n b_m - b_na_{nm} - b_ma_{mn}\right) \dd\left(\delta_xZ_{i,n}\right) \wedge \mathbf{K} \dd\left(\delta_x Z_{i,m}\right)\\
    & = \kappa_{i,0} + \triangle x\sum^s_{n=1}b_n \dd Z_{i,n}\wedge \mathbf{K}\dd\left(\delta_xZ_{i,n}\right)\,.
    \end{split} 
\end{align}
In the above, we have used the second equation in equation set \eqref{eq:stochastic RK 1-form} in the first equality,  the first equation of  \eqref{eq:stochastic RK} in the third equality, and finally the skew symmetric property of $\mathbf{K}$ and  the symplecticity condition \eqref{eq:stochastic multisymplectic cond 1} in the fourth equality. Through similar calculations by using the equations \eqref{eq:stochastic RK},  \eqref{eq:stochastic RK 1-form} and the symplecticity condition \eqref{eq:stochastic multisymplectic cond 2}, we additionally obtain
\begin{align}
    \wt{\kappa}_{i,1} &= \wt{\kappa}_{i,0} + \triangle x\sum^s_{n=1}b_n \dd Z_{i,n}\wedge\wt{\mathbf{K}} \dd\left(\delta_xZ_{i,n}\right),\\
    \omega_{1,m} &= \omega_{0,m} + \sum^r_{i=1} \dd Z_{i,m}\wedge \mathbf{M}\left(\triangle t\ob{b}_i \dd \left(\delta_t^A Z_{i,m}\right) + \triangle W \ob{\beta}_i \dd \left(\delta_t^M Z_{i,m}\right) \right).
\end{align}
Taking the exterior derivative of the defining relations \eqref{eq:stochastic RK Ham}, we obtain
\begin{align}
    &\mathbf{M} \dd\left(\delta_t^A Z_{i,m}\right) + \mathbf{K} \dd \left(\delta_x Z_{i,m}\right) = \mathbf{H}''(Z_{i,m})\dd Z_{i,m}\,,\quad \mathbf{M} \dd\left(\delta_t^M Z_{i,m}\right) + \wt{\mathbf{K}} \dd\left(\delta_x Z_{i,m}\right) = \mathbf{\wt{H}}''(Z_{i,m})\dd Z_{i,m}\,,
\end{align}
where $\mathbf{H}''$ and $\mathbf{\wt{H}}''$ are the shorthands for the Hessians, $\left(\frac{\p^2 H}{\p z^i \p z^j}\right)$ and  $\left(\frac{\p^2 \wt{H}}{\p z^i \p z^j}\right)$, respectively. Then, we have
\begin{align*}
\begin{split}
    &\sum^s_{m=1}b_m\left[\omega_{1,m} - \omega_{0,m}\right]\triangle x + \sum^r_{i=1}\ob{b}_i\left[\kappa_{i, 1} - \kappa_{i, 0} \right]\triangle t + \sum^r_{i=1}\ob{\beta}_i\left[\wt{\kappa}_{i, 1} - \wt{\kappa}_{i, 0} \right]\triangle W\\
    &= \triangle x \sum^r_{i=1} \sum^s_{m=1}b_m\dd Z_{i,m}\wedge\left( \ob{b}_i\triangle t\left( \mathbf{M}\dd\left(\delta_t^A Z_{i,m}\right) + \mathbf{K}\dd \left(\delta_x Z_{i,m}\right)\right) + \ob{\beta}_i\triangle W\left( \mathbf{M}\dd\left(\delta_t^M Z_{i,m}\right) + \wt{\mathbf{K}}\dd \left(\delta_x Z_{i,m}\right)\right) \right)\\
    &= \triangle x \sum^r_{i=1}\sum^s_{m=1}b_m \dd Z_{i,m}\wedge\left(\ob{b}_i \triangle t\mathbf{H}''\dd Z_{i,m} + \ob{\beta}_i \triangle W\mathbf{\wt{H}}''\dd Z_{i,m}\right)\\& = 0\,,
\end{split}
\end{align*}
where in the last equality we have used the symmetry property of $\mathbf{H}''$ and $\mathbf{\wt{H}}''$ to obtain the required result.
\end{proof}
\begin{remark}
    The discrete multisymplectic form conservation law can be thought as the approximation of integral of \eqref{eq:stoch structural conservation laws} in the domain $[0,\triangle x] \times [0,\triangle t]$,
    \begin{align*}
        \int^{\triangle x}_0 \left( \omega(z(x,\triangle t)) - \omega(z(x,0)) \right) \, dx &+ \int^{\triangle t}_0  \left(  \kappa(z(\triangle x, t)) - \kappa(z(0,t))\right) \, dt \\
        &+ \int^{\triangle t}_0 \left( \wt{\kappa}(z(\triangle x, t)) - \wt{\kappa}(z(0,t))\right)\circ dW_t = 0\,.
    \end{align*}
\end{remark}
\begin{remark}[Bounded Brownian increments]
    The increment of Wiener process $\triangle W$ defined in \eqref{eq:brownian increments} is unbounded. This can degrade the convergence property of the nonlinear solver in the implicit multisympletic method defined by \eqref{eq:stochastic RK} and \eqref{eq:stochastic RK Ham}. 
    However, one can replace the increments appearing in the numerical scheme \eqref{eq:stochastic RK} with a truncated random variable $\ob{\triangle W}$ defined by (e.g., \cite{MRT2002b,MRT2002a, milstein2004stochastic, Woodfield2024})
    \begin{align*}
        \ob{\triangle W} = \begin{cases} 
            U\,, & \text{if } \triangle W > U\, , \\
            \triangle W\,, & \text{if } |\triangle W| \leq U\, , \\
            -U\,,& \text{if } \triangle W < -U\, ,
        \end{cases}
    \end{align*}
    where $U>0$ is constant determined by the problem such that the mean square order of convergence is unaffected. One example of a $\triangle t$ dependent $U$ can be found to be $U(\triangle t) = \sqrt{2k|\operatorname{ln}\triangle t|}$ for $k \geq 1$.
\end{remark}
\paragraph{Discrete momentum conservation.}
We consider the special case where $H(z)$ and $\wt{H}(z)$ are quadratic in $z$, i.e., they can be written in the following form
\begin{align}
    H(z) = \left<z, \mathbf{A} z\right>, \quad \wt{H}(z) = \left<z, \wt{\mathbf{A}} z\right>\,, \label{eq:quadratic Hams}
\end{align}
for some symmetric matrices $\mathbf{A}$ and $\wt{\mathbf{A}}$. In this case, we obtain additional discrete conservation laws from the scheme defined by \eqref{eq:stochastic RK} and \eqref{eq:stochastic RK Ham} 
that correspond to discrete versions of the continuous $1$-form conservation laws \eqref{eq:stoch conservations laws}. 

\begin{theorem}
    In addition to the symplecticity conditions \eqref{eq:stochastic multisymplectic cond 1} and \eqref{eq:stochastic multisymplectic cond 2} being satisfied, when the Hamiltonians satisfy \eqref{eq:quadratic Hams},
    we have the discrete momentum conservation law 
    \begin{align}
        \sum^s_{m=1}b_m\left[I_{1,m} - I_{0,m}\right]\triangle x + \sum^r_{i=1}\ob{b}_i\left[G_{i,1} - G_{i,0} \right]\triangle t + \sum^r_{i=1}\ob{\beta}_i\left[\wt{G}_{i,1} - \wt{G}_{i,0} \right]\triangle W = 0\,. \label{eq:discrete momentum conserversion law}
    \end{align}
    In equation \eqref{eq:discrete momentum conserversion law}, we have defined 
    \begin{equation}
    \begin{aligned}
        I_{i,m} &= \frac{1}{2}\scp{\mathbf{M}  \left(\delta_x z_{i,m}\right)}{z_{i,m}}\, ,\\
        G_{i,m} &= H(z_{i,m}) - \frac{1}{2}\scp{\mathbf{M} \left(\delta^A_t z_{i,m}\right)}{z_{i,m}} \,, \\
        \wt{G}_{i,m} &= \wt{H}(z_{i,m}) - \frac{1}{2}\scp{\mathbf{M} \left(\delta^M_t z_{i,m}\right)}{z_{i,m}}\,,
    \end{aligned}
    \end{equation}
    where $\scp{\cdot}{\cdot}$ denotes the inner product of the vectors and it is not integrated over the domain.
\end{theorem}
\begin{proof}
    The proof is similar to that of discrete multisymplectic form formula \eqref{eq:discrete stoch ms CL}. Using \eqref{eq:stochastic RK}, multisymplectic conditions \eqref{eq:stochastic multisymplectic cond 1} and \eqref{eq:stochastic multisymplectic cond 2}, we obtain
    \begin{align*}
        \begin{split}
            G_{i,1} &= \scp{z_{i, 0} + \triangle x\sum^s_{n=1}b_{n}\left(\delta_xZ_{i,n}\right)}{\mathbf{A}\left(z_{i, 0} + \triangle x\sum^s_{n=1}b_{n}\left(\delta_xZ_{i,n}\right)\right)} \\
            & \qquad - \frac{1}{2}\scp{\mathbf{M}\delta^A_t\left(z_{i, 0} + \triangle x\sum^s_{n=1}b_{n}\left(\delta_xZ_{i,n}\right)\right)}{z_{i, 0} + \triangle x\sum^s_{n=1}b_{n}\left(\delta_xZ_{i,n}\right)}\\
            &=G_{i,0} + \triangle x\sum^s_{n=1}b_{n}\left[\scp{z_{i,0}}{\mathbf{A} \left(\delta_x Z_{i,n}\right)} + \scp{\delta_x Z_{i,n}}{\mathbf{A} z_{i,0}}\right] + (\triangle x)^2\sum^s_{n=1}\sum^s_{m=1}b_{n}b_{m}\scp{\delta_xZ_{i,n}}{\mathbf{A}\left(\delta_xZ_{i,m}\right)}\\
            & \qquad - \frac{\triangle x}{2}\sum^s_{n=1}b_{n}\left[\scp{\mathbf{M} \left(\delta^A_t z_{i,0}\right)}{\delta_x Z_{i,n}} + \scp{\mathbf{M}\delta^A_t\left(\delta_x Z_{i,n}\right)}{ z_{i,0}}\right] \\
            &\qquad - \frac{(\triangle x)^2}{2}\sum^s_{n=1}\sum^s_{m=1}b_{n}b_{m}\scp{\mathbf{M}\delta^A_t\left(\delta_xZ_{i,n}\right)}{\delta_xZ_{i,m}}\\
            &=G_{i,0} + 2\triangle x\sum_{n=1}^sb_n\scp{Z_{i,n}}{\mathbf{A}\left(\delta_x Z_{i,n}\right)} - \frac{\triangle x}{2}\sum_{n=1}^sb_n\left[\scp{\mathbf{M}\left(\delta^A_tZ_{i,n}\right)}{\delta_x Z_{i,n}} + \scp{\mathbf{M}\delta^A_t\left(\delta_xZ_{i,n}\right)}{Z_{i,n}}\right]\,.
        \end{split}
    \end{align*}
    Similarly, we have the following expressions
\begin{align}
\begin{split}
    \wt{G}_{i,1} &= \wt{G}_{i,0} + 2\triangle x\sum_{n=1}^sb_n\scp{Z_{i,n}}{\wt{\mathbf{A}}\left(\delta_x Z_{i,n}\right)} \\
    &\qquad \quad - \frac{\triangle x}{2}\sum_{n=1}^sb_n\left[\scp{\mathbf{M} \left(\delta^M_tZ_{i,n}\right)}{\delta_x Z_{i,n}} + \scp{\mathbf{M}\delta^M_t\left(\delta_xZ_{i,n}\right)}{Z_{i,n}}\right]\,,\\
    I_{1,m} &= I_{0,m} + \frac{1}{2}\sum^r_{i=1} \scp{\mathbf{M}\delta_x\left(\triangle t \ob{b}_i \left(\delta^A_t Z_{i,m}\right) + \triangle W \ob{\beta}_i \left(\delta^M_t Z_{i,m}\right)\right)}{Z_{i,m}} \\
    & \qquad \quad + \frac{1}{2}\sum^r_{i=1} \scp{\mathbf{M}\left(\delta_xZ_{i,m}\right)}{\triangle t \ob{b}_i \left(\delta^A_t Z_{i,m}\right) + \triangle W \ob{\beta}_i \left(\delta^M_t Z_{i,m}\right)}\,.
\end{split}
\end{align}
Then we have 
\begin{align*}
    \begin{split}
        &\sum^s_{m=1}b_m\left[I_{1,m} - I_{0,m}\right]\triangle x + \sum^r_{i=1}\ob{b}_i\left[G_{i,1} - G_{i,0} \right]\triangle t + \sum^r_{i=1}\ob{\beta}_i\left[\wt{G}_{i,1} - \wt{G}_{i,0} \right]\triangle W \\
        &=\sum^s_{m=1}\sum^r_{i=1}b_m\triangle x \Big[\left(\scp{\triangle t\ob{b}_i\nabla_z H(Z_{i,m})}{\delta_x Z_{i,m}} + \scp{\triangle W\ob{\beta}_i\nabla_z \wt{H}(Z_{i,m})}{\delta_x Z_{i,m}}\right) \\
        & \qquad \qquad\qquad\qquad  - \frac{1}{2}\triangle t \ob{b}_i\scp{\mathbf{M}\left(\delta^A_tZ_{i,m}\right)}{\delta_x Z_{i,m}} - \frac{1}{2}\triangle W \ob{\beta}_i\scp{\mathbf{M}\left(\delta^M_tZ_{i,m}\right)}{\delta_x Z_{i,m}}  \\
        &\qquad \qquad\qquad \qquad+ \frac{1}{2}\scp{\mathbf{M}\left(\delta_xZ_{i,m}\right)}{\triangle t \ob{b}_i \left(\delta^A_t Z_{i,m}\right) + \triangle W \ob{\beta}_i \left(\delta^M_t Z_{i,m}\right)}  \Big]\\
        &=\sum^s_{m=1}\sum^r_{i=1}b_m\triangle x \left\langle\triangle t\ob{b}_i\nabla_z H(Z_{i,m}) + \triangle W\ob{\beta}_i\nabla_z \wt{H}(Z_{i,m})\right. \\
        &\qquad \qquad\qquad \qquad\left.  - \triangle t \ob{b}_i\mathbf{M}\left(\delta^A_tZ_{i,m}\right) - \triangle W \ob{\beta}_i\mathbf{M} \left(\delta^M_tZ_{i,m}\right){\delta_x Z_{i,m}}\right\rangle\\
        &=\sum^s_{m=1}\sum^r_{i=1}b_m\triangle x \scp{\triangle t\ob{b}_i\mathbf{K}\left(\delta_xZ_{i,m} \right)+ \triangle W\ob{b}_i\mathbf{\wt{K}}\left(\delta_xZ_{i,m}\right)}{\delta_x Z_{i,m}}\\
        &= 0\,.
    \end{split}
\end{align*}
Here, we have used the commutative property of $\delta_x$ with $\delta^A_t$ and $\delta^M_t$; the defining relation \eqref{eq:stochastic RK Ham} in the third equality and the skew-symmetry property of matrices $\mathbf{K}$ and $\wt{\mathbf{K}}$ in the last equality. 
\end{proof}

\paragraph{Implicit midpoint scheme.} The most well-known examples of symplectic Runge--Kutta methods which satisfy the symplecticity conditions \eqref{eq:stochastic multisymplectic cond 1} and \eqref{eq:stochastic multisymplectic cond 2} include the deterministic implicit midpoint scheme \cite{LZ2011, MW2001} and the stochastic implicit midpoint scheme \cite{HT2018b, ZZHS2017}, respectively. Here, we have $s = r = 1$ and we represent the various coefficients appearing in the general form of the stochastic collocation method \eqref{eq:stochastic RK} in Butcher tableau form, 
\begin{align}
    \begin{split}
        \begin{array}{c|c}
        c_1 & a_{11}  \\
        \hline
         \rule{0pt}{1.em}& b_{1} 
        \end{array} \quad 
        = \quad
        \begin{array}{c|c}
        d_1 & \ob{a}_{11}  \\
        \hline
         \rule{0pt}{1.em}& \ob{b}_1 
        \end{array} \quad 
        = \quad 
        \begin{array}{c|c}
        d_1 & \ob{\alpha}_{11}  \\
        \hline
         \rule{0pt}{1.em}& \ob{\beta}_1 
        \end{array} \quad 
        = \quad
        \begin{array}{c|c}
        1/2 & 1/2  \\
        \hline
         \rule{0pt}{1.em}& 1  
        \end{array}\,.
    \end{split} \label{eq:IM SRK coeffs}
\end{align}
We introduce concrete discrete differencing operators, 
\begin{align*}
    \delta_t z_{i,m} = \delta^A_t z_{i,m} = \delta^M_t z_{i,m} = \frac{z_{i+1,m} - z_{i,m}}{\triangle t}\,, \quad \delta_x z_{i,m} = \frac{z_{i,m+1} - z_{i,m}}{\triangle x}\,,
\end{align*}
and for notational purposes introduce discrete averaging operators
\begin{align*}
    \mathcal{A}_t z_{i,m} = \frac{1}{2}(z_{i+1,m} + z_{i,m})\,, \quad \mathcal{A}_x z_{i,m} = \frac{1}{2}(z_{i,m+1} + z_{i,m})\,,
\end{align*}
such that the stochastic multisymplectic Runge--Kutta method defined by \eqref{eq:stochastic RK} and \eqref{eq:stochastic RK Ham} can be written as
\begin{align}
\begin{split}
    &\mathbf{M}\delta_t\mathcal{A}_x (z_{i,m}) + \mathbf{K}\delta_x \mathcal{A}_t (z_{i,m}) + 
    \mathbf{\wt{K}}\delta_x\mathcal{A}_t (z_{i,m})\frac{\triangle W}{\triangle t} = \nabla_z H(\mathcal{A}_t\mathcal{A}_x (z_{i,m})) + \nabla_z \wt{H}(\mathcal{A}_t\mathcal{A}_x (z_{i,m}))\frac{\triangle W}{\triangle t}\,.
\end{split}\label{eq:stochastic MS IM}
\end{align}

\begin{remark}
    In notation more familiar in numerical analysis of midpoint schemes, we define
    \begin{align*}
        z_{i+1/2,m} = \mathcal{A}_t z_{i,m}\,, \quad z_{i,m+1/2} = \mathcal{A}_x z_{i,m}\,, \quad z_{i+1/2,m+1/2} = \mathcal{A}_t \mathcal{A}_x z_{i,m}\,,
    \end{align*}
    and the stochastic multisymplectic Runge--Kutta method \eqref{eq:stochastic MS IM} can be equivalently expressed as
    \begin{align*}
    &\mathbf{M}\left(\frac{z_{i+1,m+\frac{1}{2}} - z_{i,m+\frac{1}{2}}}{\triangle t}\right) + \mathbf{K}\left(\frac{z_{i+\frac{1}{2},m+1} - z_{i+\frac{1}{2},m}}{\triangle x}\right) + 
        \mathbf{\wt{K}}\left(\frac{z_{i+\frac{1}{2},m+1} - z_{i+\frac{1}{2},m}}{\triangle x}\right)\frac{\triangle W}{\triangle t} \\
        & \qquad = \nabla_z H\left(z_{i+\frac{1}{2},m+\frac{1}{2}}\right) + \nabla_z \wt{H}\left(z_{i+\frac{1}{2},m+\frac{1}{2}}\right)\frac{\triangle W}{\triangle t}\,.
    \end{align*}
\end{remark}


\subsection{Numerical results: Stochastic NLS equations}
\label{subsec:NLS numerics}
In this subsection, we will numerically investigate the application of stochastic multisymplectic implicit midpoint scheme \eqref{eq:stochastic MS IM} to the classical example of NLS equations, under the two example stochastic perturbations defined by equations \eqref{eq:SALT NLS} and \eqref{eq:stoch NLS}, respectively. Recall that their multisymplectic formulations are given by \eqref{eq:SALT multi NLS} and \eqref{eq:stoch multi NLS}, respectively. 

\paragraph{Stochastic transport.}
For the stochastic transport case, we apply the stochastic multisymplectic implicit midpoint scheme \eqref{eq:stochastic MS IM} to the system  \eqref{eq:SALT multi NLS} and we obtain the following discrete equations
\begin{align}
    \begin{split}
        &\delta_t\mathcal{A}_xp_{i,m} + \mathcal{A}_t\delta_xw_{i,m} + \xi \mathcal{A}_t\delta_xp_{i,m} \frac{\triangle W}{\triangle t} = 2\kappa\left((\mathcal{A}_x\mathcal{A}_tp_{i,m})^2 + (\mathcal{A}_x\mathcal{A}_tq_{i,m})^2\right)\mathcal{A}_x\mathcal{A}_tq_{i,m}\,,\\
        &\delta_t\mathcal{A}_xq_{i,m} - \mathcal{A}_t\delta_xv_{i,m} + \xi \mathcal{A}_t\delta_xq_{i,m} \frac{\triangle W}{\triangle t} = -2\kappa\left((\mathcal{A}_x\mathcal{A}_tp_{i,m})^2 + (\mathcal{A}_x\mathcal{A}_tq_{i,m})^2\right)\mathcal{A}_x\mathcal{A}_tp_{i,m}\,,\\
        &\mathcal{A}_t\delta_xp_{i,m} = \mathcal{A}_t\mathcal{A}_xv_{i,m} \,,\\ &\mathcal{A}_t\delta_xq_{i,m}  = \mathcal{A}_t\mathcal{A}_xw_{i,m} \,.
    \end{split}\label{eq:IM RK SALT}
\end{align}
The third and fourth equations of \eqref{eq:IM RK SALT} imply the discrete constraints
\begin{align}
    \delta_xp_{i,m} = \mathcal{A}_xv_{i,m}\,\quad \text{and} \quad \delta_xq_{i,m} = \mathcal{A}_xw_{i,m}\,, \label{eq:discrete constraints}
\end{align}
which can be inserted back to the first and second equations of \eqref{eq:IM RK SALT} to have the stochastic integrator for \eqref{eq:SALT NLS},
\begin{align}
    \begin{split}
        &\delta_t\mathcal{A}_x^2p_{i,m} + \mathcal{A}_t\delta_x^2q_{i,m} + \xi \mathcal{A}_t\mathcal{A}_x\delta_xp_{i,m} \frac{\triangle W}{\triangle t} = 2\kappa \mathcal{A}_x\left(\left((\mathcal{A}_x\mathcal{A}_tp_{i,m})^2 + (\mathcal{A}_x\mathcal{A}_tq_{i,m})^2\right)\mathcal{A}_x\mathcal{A}_tq_{i,m}\right)\,,\\
        &\delta_t\mathcal{A}_x^2q_{i,m} - \mathcal{A}_t\delta_x^2p_{i,m} + \xi \mathcal{A}_t\mathcal{A}_x\delta_xq_{i,m} \frac{\triangle W}{\triangle t} = -2\kappa \mathcal{A}_x\left(\left((\mathcal{A}_x\mathcal{A}_tp_{i,m})^2 + (\mathcal{A}_x\mathcal{A}_tq_{i,m})^2\right)\mathcal{A}_x\mathcal{A}_tp_{i,m}\right)\,.
    \end{split}\label{eq:sms SALT method NLS}
\end{align}
We verify our numerical scheme by considering a stochastic soliton solution to the stochastic transport NLS equation \eqref{eq:SALT NLS}. The setup of the numerical experiment is as follows. For the spatial discretisation, we consider a periodic domain $\mcal{D} = [0,40]$ with $\triangle x = 0.1$. For the temporal discretisation, we take $t_0 = 0$ and $t_1 = 80$ with $\triangle t = 0.02$. For single bright soliton solutions, we take $\kappa = -1$ and start the simulation with initial conditions 
\begin{align}
    p(x, 0) = \frac{1}{\sqrt{2}}\operatorname{sech}\left(\frac{1}{\sqrt{2}}(x-15)\right)\operatorname{cos}\left(\frac{x}{20}\right)\,,\quad q(x, 0) = -\frac{1}{\sqrt{2}}\operatorname{sech}\left(\frac{1}{\sqrt{2}}(x-15)\right)\operatorname{sin}\left(\frac{x}{20}\right)\,,\label{eq:nls init}
\end{align}
such that we have the exact solution 
\begin{align*}
    p(x, t) = A(\wt{x},t)\operatorname{cos}\left(\frac{\wt{x}}{20} + \frac{199t}{400}\right) \,,\quad q(x, t) = - A(\wt{x},t)\operatorname{sin}\left(\frac{\wt{x}}{20} + \frac{199t}{400}\right)\,,
\end{align*}
where $A(x,t) = \frac{1}{\sqrt{2}}\operatorname{sech}\left(\frac{1}{\sqrt{2}}\left(x-\frac{t}{10}-15\right)\right)$ and $\wt{x}(t) := x - \xi W(t)$ for $W(t) = \int_0^t dW_s$. Taking $\xi = 0.1$, the solution behaviour of one single realisation is shown in Figure \ref{fig:SALT NLS single realisation}. From the evolution profile, one can see the stochastic motion of the soliton which keeps its form as it moves to the right. In the simulation, the absolute tolerance of the nonlinear solver in the implicit scheme is set to $10^{-6}$. This tolerance is reflected in the error of the global conservation law which is $\mcal{O}(10^{-6})$.
\begin{figure}[!ht]
    \centering
    \begin{subfigure}[b]{0.32\textwidth}
        \centering
        \includegraphics[width=\textwidth]{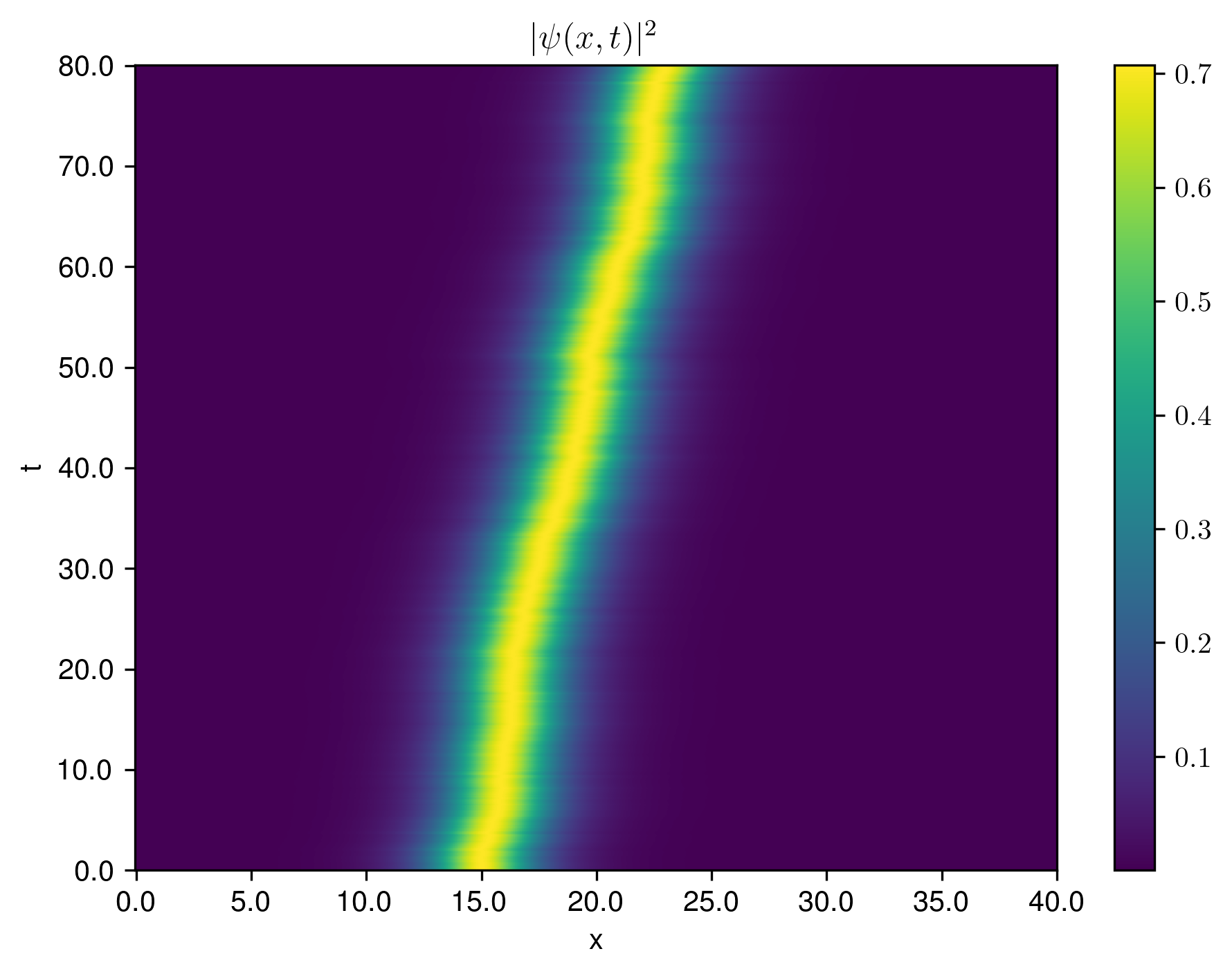}
    \end{subfigure}
    \begin{subfigure}[b]{0.32\textwidth}
        \centering
        \includegraphics[width=\textwidth]{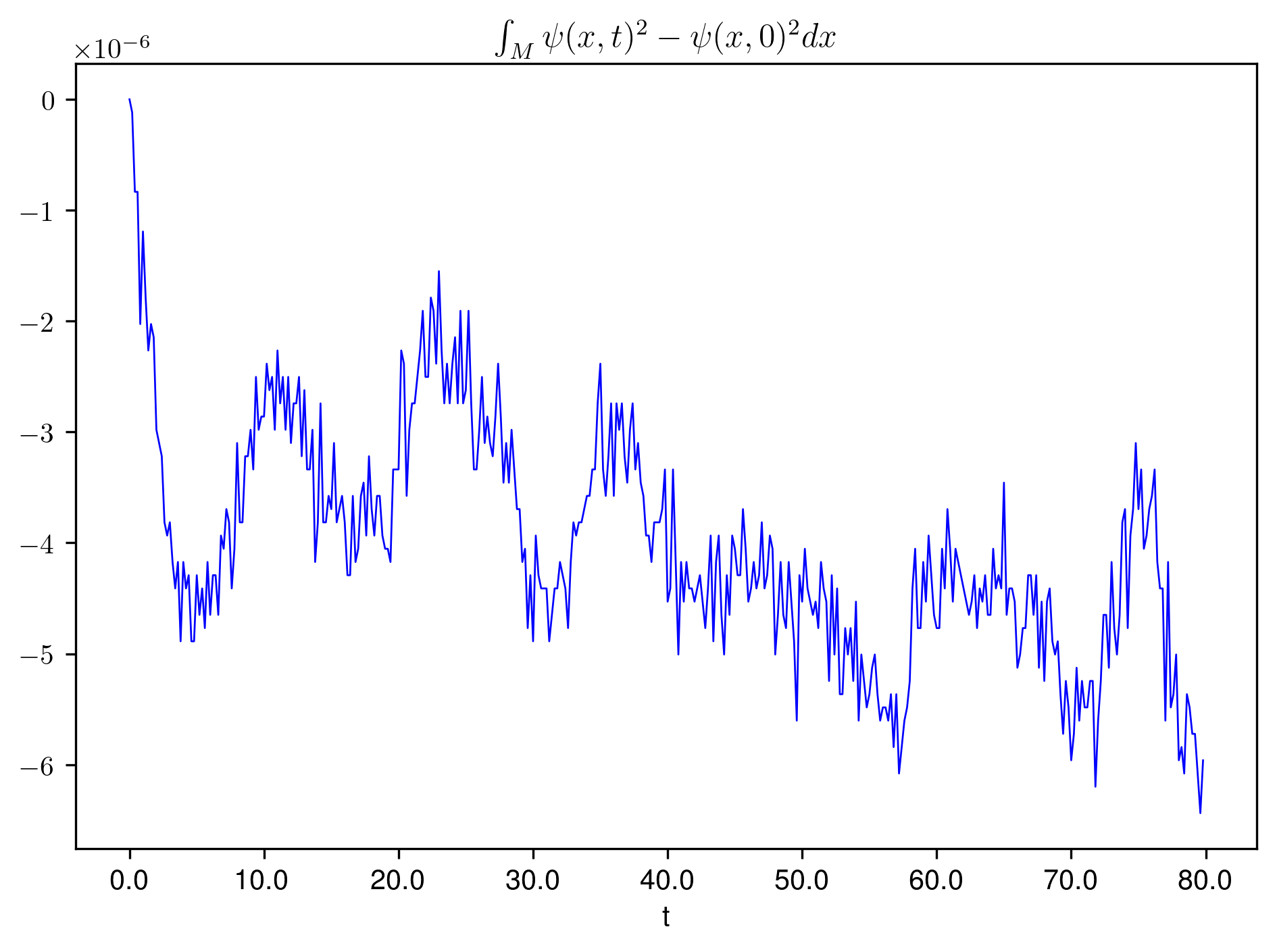}
    \end{subfigure}
    \begin{subfigure}[b]{0.32\textwidth}
        \centering
        \includegraphics[width=\textwidth]{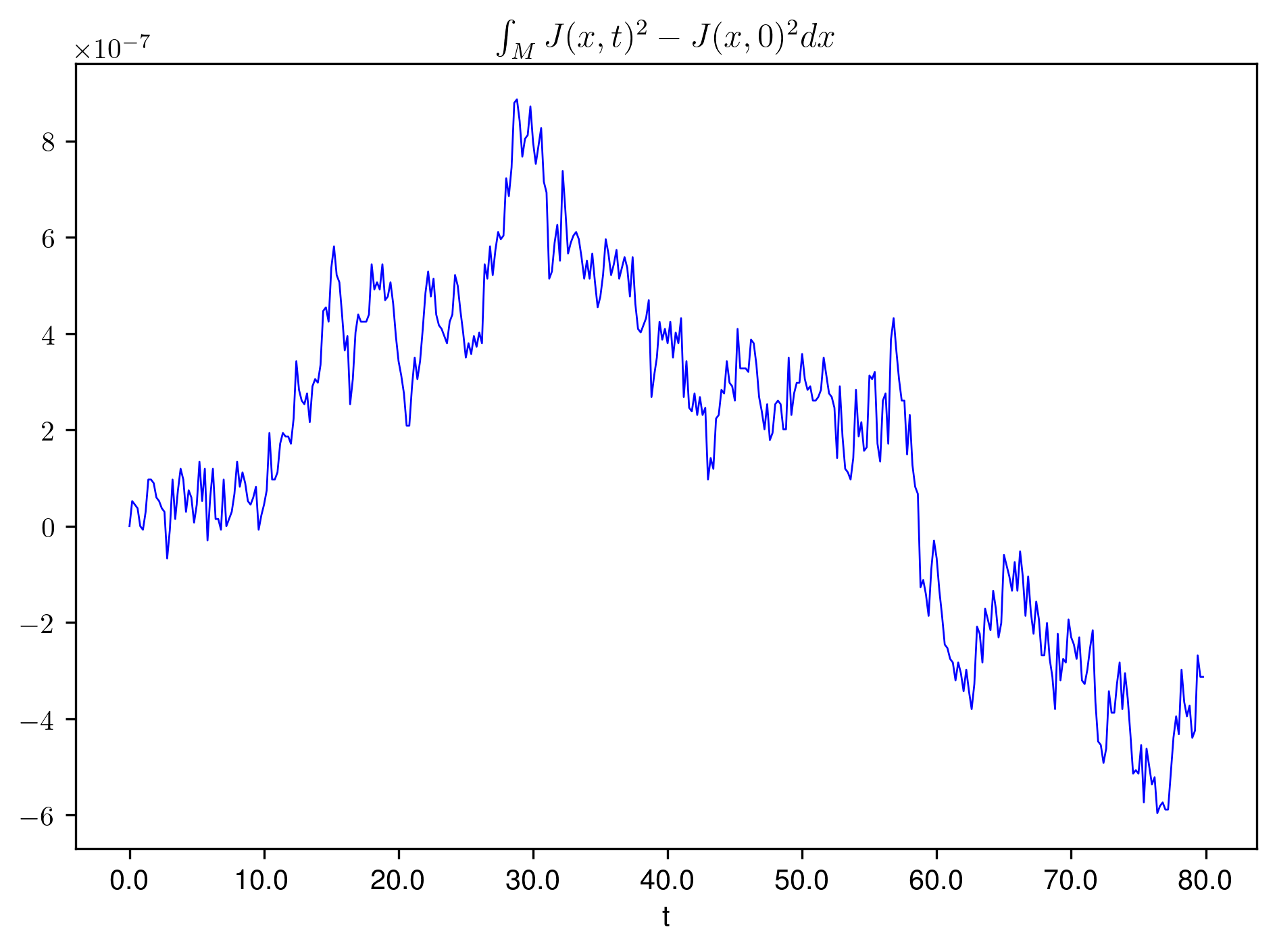}
    \end{subfigure}
    \caption{Evolution profile of $|\psi(x,t)|^2$ (left), the error of the global (domain integrated) conservation of wave density $|\psi(x,t)|^2$ (middle), and the error of the global conservation of wave momentum $J = \psi \nabla \psi^*$ (right) for the NLS equation with stochastic transport \eqref{eq:SALT NLS}, driven by a single realisation of Brownian motion.}
    \label{fig:SALT NLS single realisation}
\end{figure}

Additionally, we evaluate the ensemble statistics of the numerical scheme by considering an ensemble with $32$ members, all with the setup and parameters fixed as in the single realisation case. The results are shown in Figure \ref{fig:SALT NLS ensemble}. From the ensemble errors in global conservation of wave density $|\psi(x,t)|^2$, one sees that the ensemble standard deviation increases with $t$, while the ensemble mean error contains a drift. This is in contrast to the errors in global conservation of wave momentum, whose the ensemble mean oscillates around $0$.
\begin{figure}[!ht]
    \centering
    \begin{subfigure}[b]{0.45\textwidth}
        \centering
        \includegraphics[width=\textwidth]{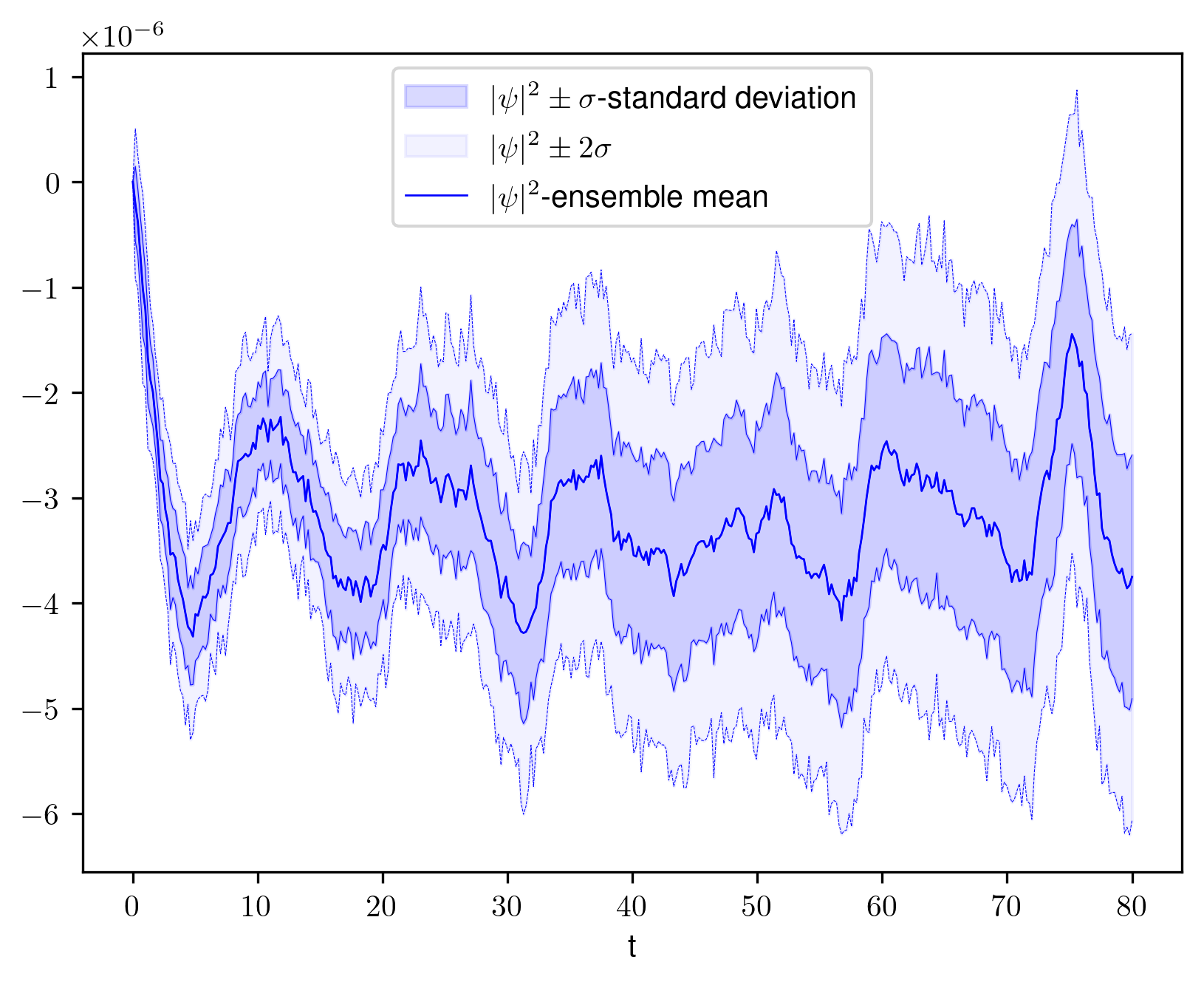}
    \end{subfigure}
    \begin{subfigure}[b]{0.45\textwidth}
        \centering
        \includegraphics[width=\textwidth]{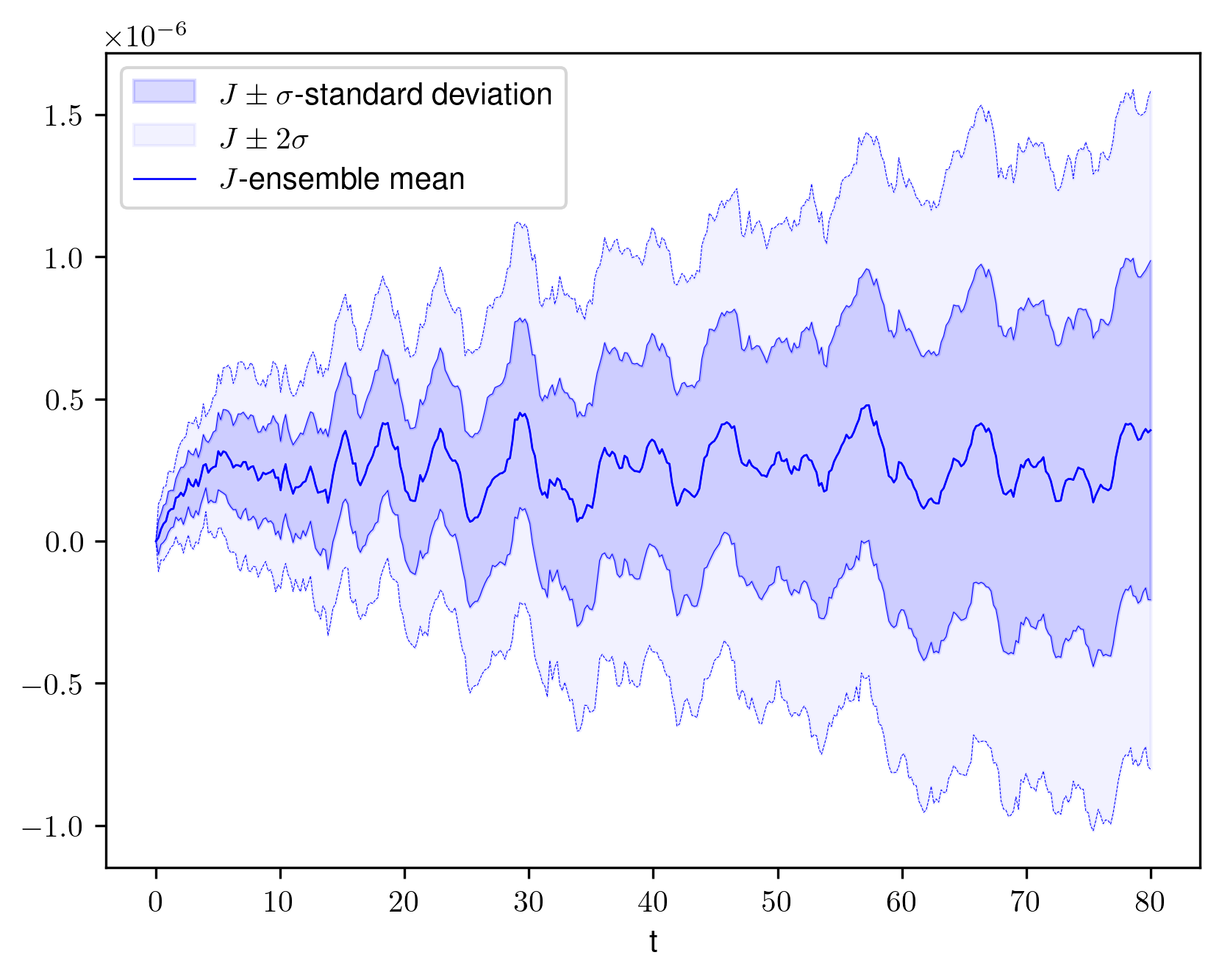}
    \end{subfigure}
    \caption{The ensemble mean and standard deviation of the errors in the global conservation of wave density $|\psi(x,t)|^2$ (left) and the global conservation of wave momentum $J = \psi \nabla \psi^*$ (right) for an ensemble of $32$ members, generated by $32$ independent realisation of Brownian motion in the NLS equation with stochastic transport \eqref{eq:SALT NLS}.}
    \label{fig:SALT NLS ensemble}
\end{figure}

\paragraph{Stochastic dispersion.}
For the stochastic dispersion case, we apply the stochastic multisymplectic implicit midpoint scheme \eqref{eq:stochastic MS IM} to the multisymplectic system \eqref{eq:stoch multi NLS} and obtain the following discrete equations
\begin{align}
    \begin{split}
        &\delta_t\mathcal{A}_xp_{i,m} + \mathcal{A}_t\delta_xw_{i,m} + \e \mathcal{A}_t\delta_xw_{i,m} \frac{\triangle W}{\triangle t} = 2\kappa\left((\mathcal{A}_x\mathcal{A}_tp_{i,m})^2 + (\mathcal{A}_x\mathcal{A}_tq_{i,m})^2\right)\mathcal{A}_x\mathcal{A}_tq_{i,m}\,,\\
        &\delta_t\mathcal{A}_xq_{i,m} - \mathcal{A}_t\delta_xv_{i,m} - \e \mathcal{A}_t\delta_xv_{i,m} \frac{\triangle W}{\triangle t} = -2\kappa\left((\mathcal{A}_x\mathcal{A}_tp_{i,m})^2 + (\mathcal{A}_x\mathcal{A}_tq_{i,m})^2\right)\mathcal{A}_x\mathcal{A}_tp_{i,m}\,,\\
        &\mathcal{A}_t\delta_xp_{i,m} + \e \mathcal{A}_t\delta_xp_{i,m}\frac{\triangle W}{\triangle t} = \mathcal{A}_t\mathcal{A}_xv_{i,m} + \e \mathcal{A}_t\mathcal{A}_xv_{i,m}\frac{\triangle W}{\triangle t}\,,\\
        &\mathcal{A}_t\delta_xq_{i,m} + \e \mathcal{A}_t\delta_xq_{i,m}\frac{\triangle W}{\triangle t} = \mathcal{A}_t\mathcal{A}_xw_{i,m} + \e \mathcal{A}_t\mathcal{A}_xw_{i,m}\frac{\triangle W}{\triangle t}\,.
    \end{split}\label{eq:IM RK NLS}
\end{align}
The third and fourth equations of \eqref{eq:IM RK NLS} imply the same constraints as that appearing in \eqref{eq:discrete constraints} which can be inserted back to the first and second equations of \eqref{eq:IM RK NLS} to have the stochastic integrator for \eqref{eq:stoch NLS},
\begin{align}
    \begin{split}
        &\delta_t\mathcal{A}_x^2p_{i,m} + \mathcal{A}_t\delta_x^2q_{i,m} + \e \mathcal{A}_t\delta_x^2q_{i,m} \frac{\triangle W}{\triangle t} = 2\kappa \mathcal{A}_x\left(\left((\mathcal{A}_x\mathcal{A}_tp_{i,m})^2 + (\mathcal{A}_x\mathcal{A}_tq_{i,m})^2\right)\mathcal{A}_x\mathcal{A}_tq_{i,m}\right)\,,\\
        &\delta_t\mathcal{A}_x^2q_{i,m} - \mathcal{A}_t\delta_x^2p_{i,m} - \e \mathcal{A}_t\delta_x^2p_{i,m} \frac{\triangle W}{\triangle t} = -2\kappa \mathcal{A}_x\left(\left((\mathcal{A}_x\mathcal{A}_tp_{i,m})^2 + (\mathcal{A}_x\mathcal{A}_tq_{i,m})^2\right)\mathcal{A}_x\mathcal{A}_tp_{i,m}\right)\,.
    \end{split}
\end{align}
For the numerical simulations, we consider the same parameters with the same bright soliton initial conditions \eqref{eq:nls init}. Here, we set $\e = 0.02$ and obtain Figures \ref{fig:stoch dispersion 1} from simulation with a single realisation. From the evolution profile, one can see that due to the stochastic nature of the dispersion parameter, the solitary wave deforms during the motion and develops small-scale structures.  

\begin{figure}[!ht]
    \centering
    \begin{subfigure}[b]{0.32\textwidth}
        \centering
        \includegraphics[width=\textwidth]{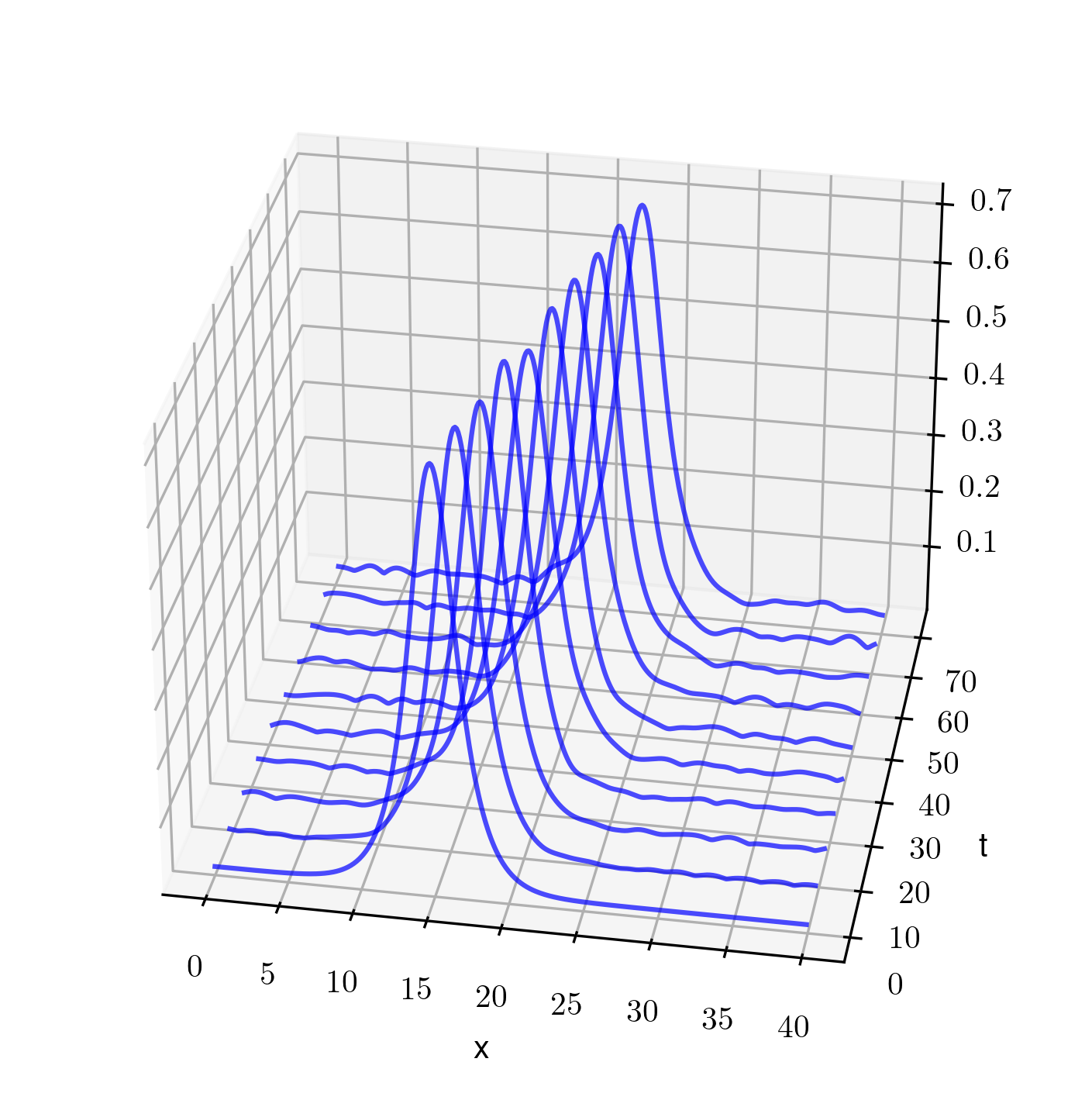}
    \end{subfigure}
    \begin{subfigure}[b]{0.32\textwidth}
        \centering
        \includegraphics[width=\textwidth]{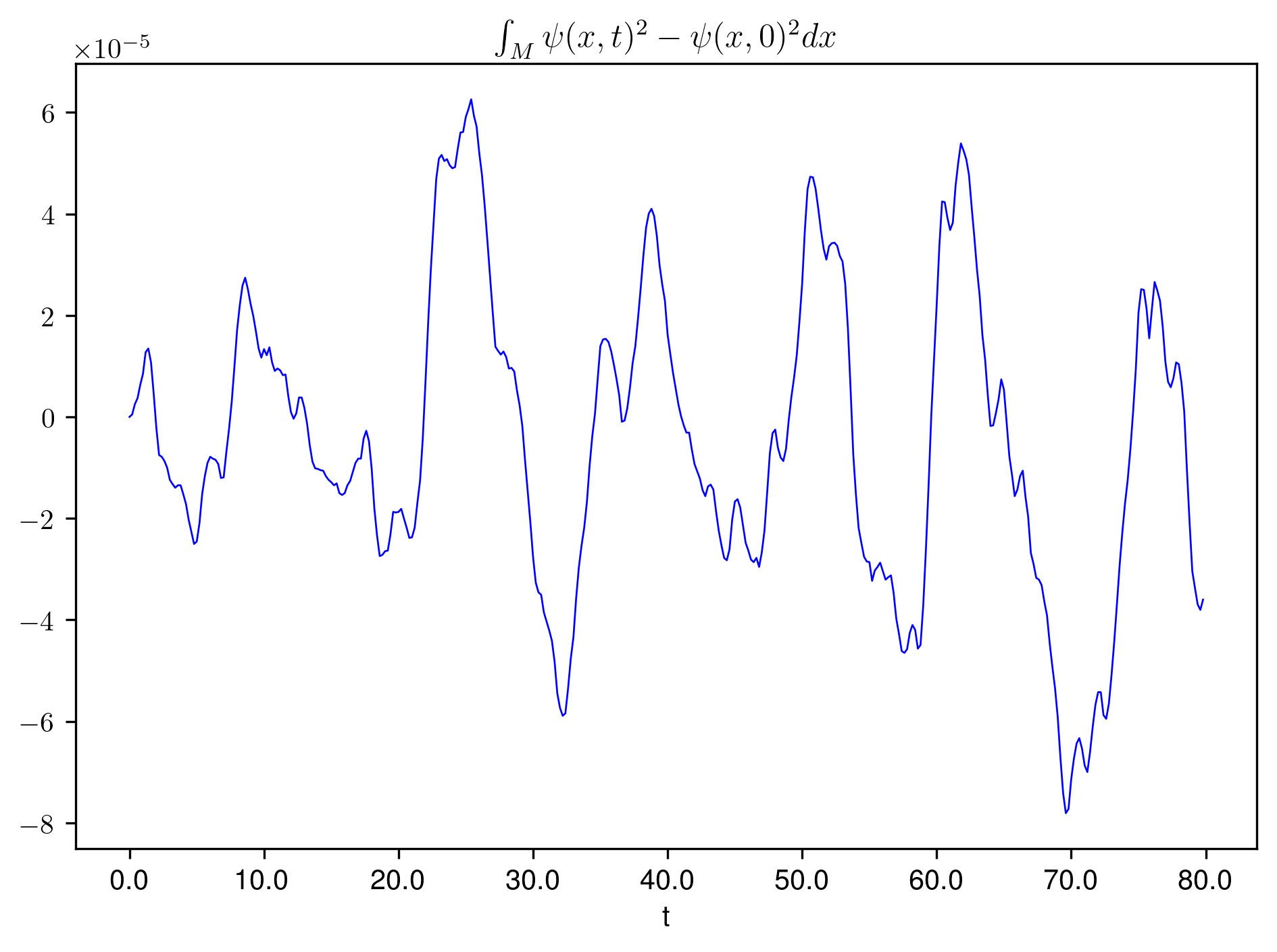}
    \end{subfigure}
    \begin{subfigure}[b]{0.32\textwidth}
        \centering
        \includegraphics[width=\textwidth]{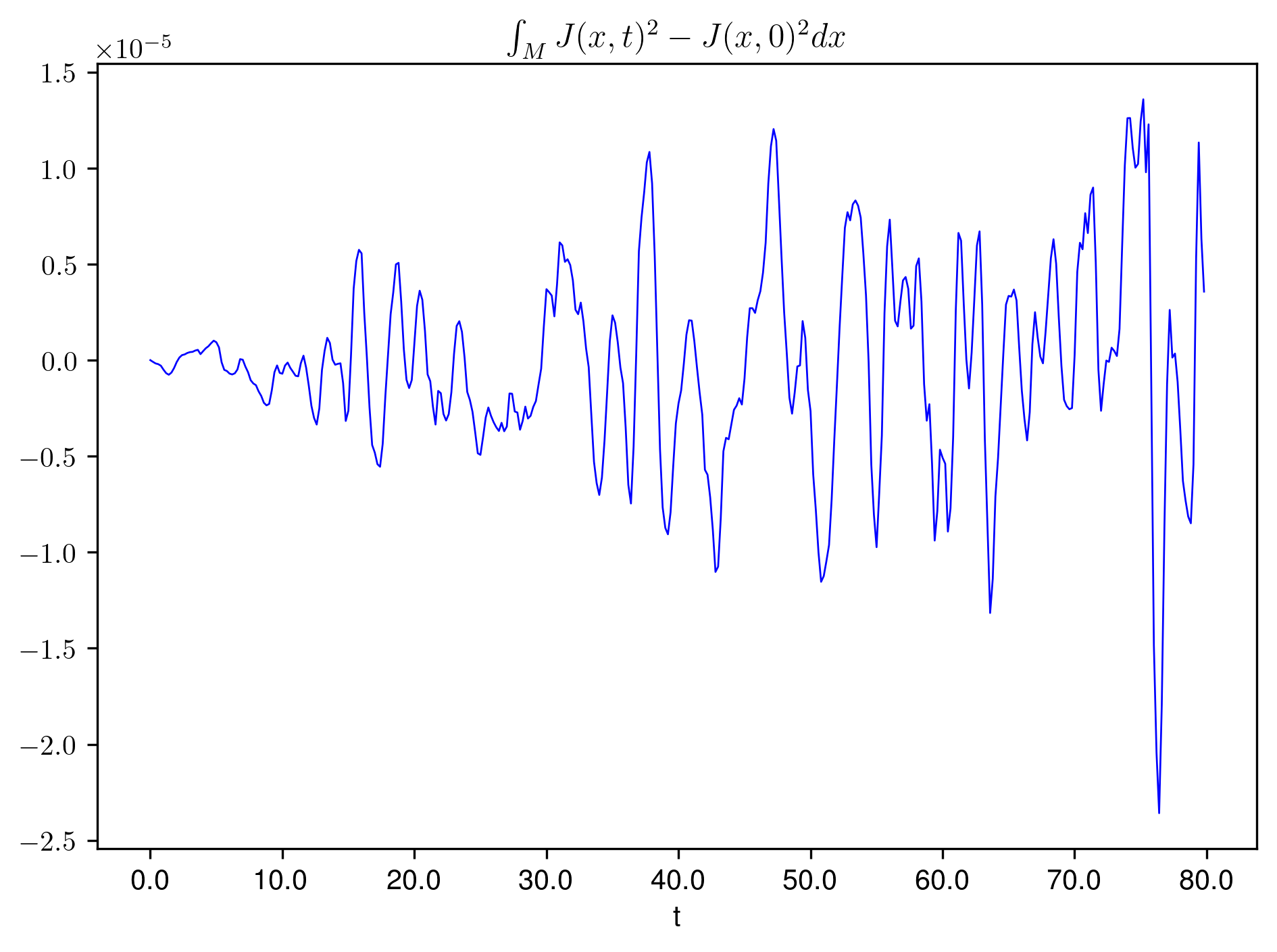}
    \end{subfigure}
        \caption{Evolution profile of $|\psi(x,t)|^2$ (left), the error of the global conservation of wave density $|\psi(x,t)|^2$ (middle), and the error of the global conservation of wave momentum $J = \psi \nabla \psi^*$ (right) for the NLS equation with stochastic dispersion \eqref{eq:stoch NLS}, driven by a single realisation of Brownian motion.}
    \label{fig:stoch dispersion 1}
\end{figure}

Similar to the stochastic transport case, we evaluate the ensemble statistics of the numerical scheme by considering an ensemble with $32$ members, all with the setup and parameters fixed as in the single realisation case. The results are shown in Figure \ref{fig:stoch dispersion 2}. Here, both the ensemble errors in global conservation of wave density and wave momentum have their ensemble mean oscillating around $0$ with increasing variance. In comparison with the ensemble statistics in Figure \ref{fig:SALT NLS ensemble}, the error ensemble variance is larger with stochastic dispersion than stochastic transport. 
\begin{figure}[!ht]
\centering
    \begin{subfigure}[b]{0.45\textwidth}
        \centering
        \includegraphics[width=\textwidth]{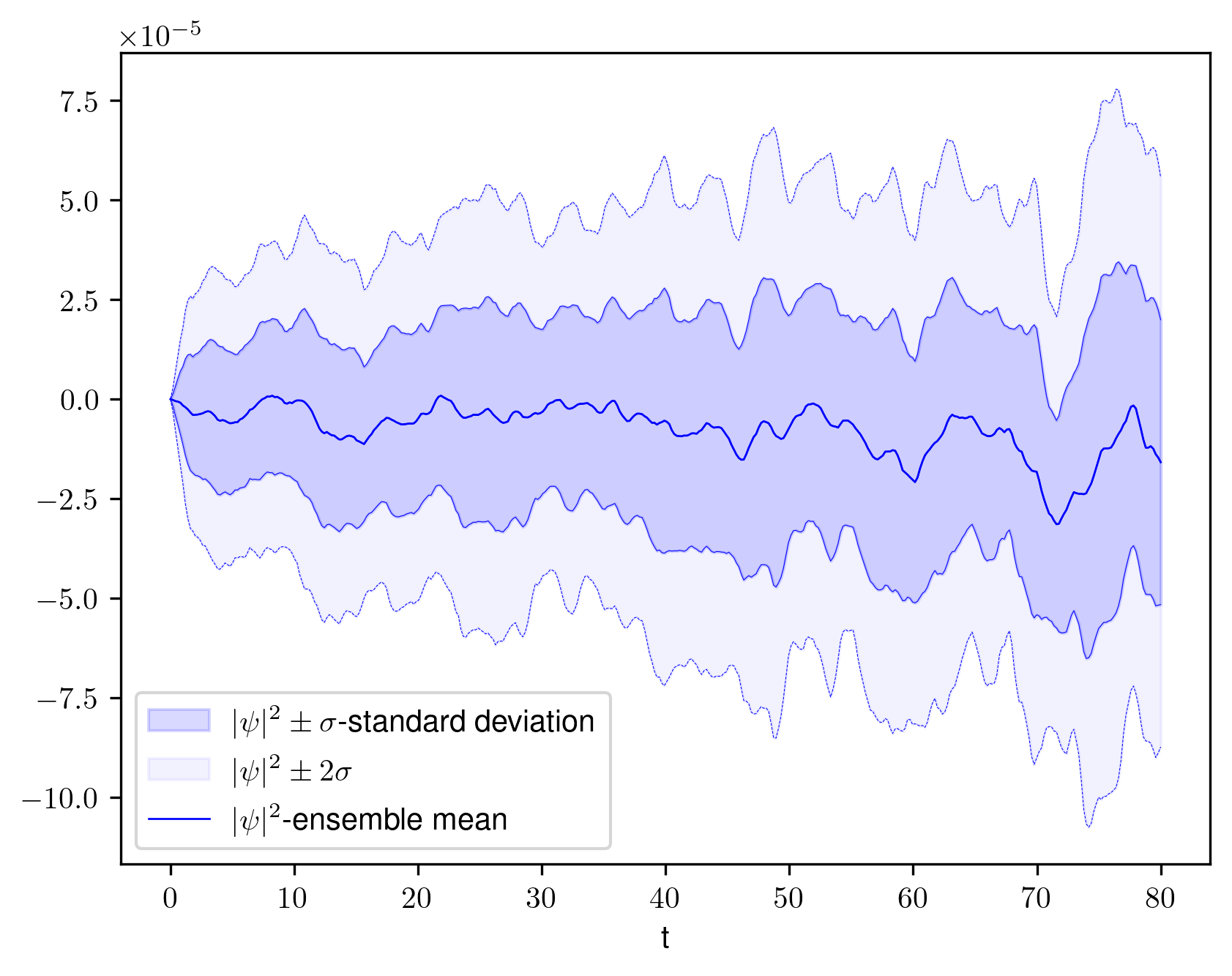}
    \end{subfigure}
    \begin{subfigure}[b]{0.45\textwidth}
        \centering
        \includegraphics[width=\textwidth]{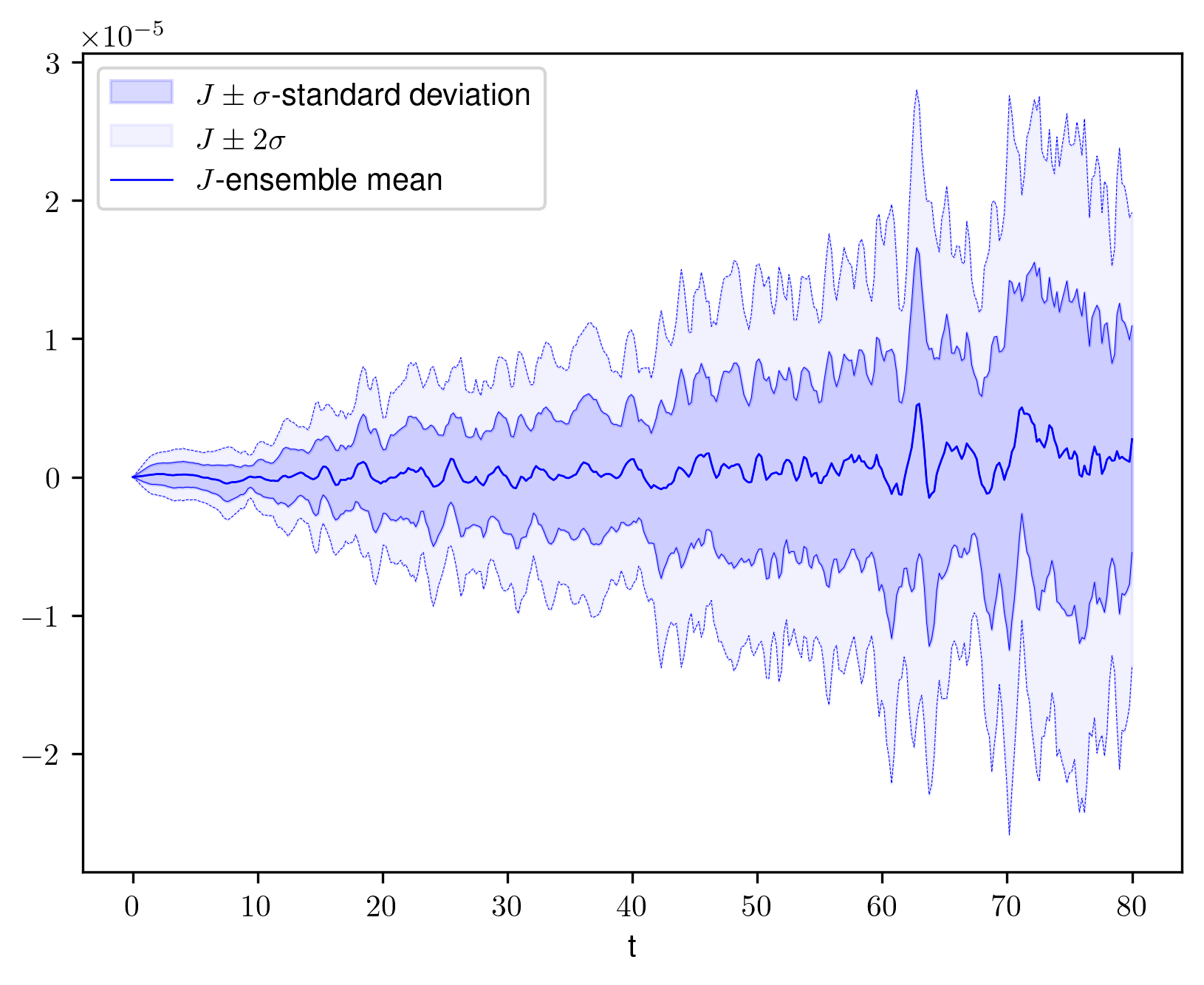}
    \end{subfigure}
    \caption{The ensemble mean and standard deviation of the errors in the global conservation of wave density $|\psi(x,t)|^2$ (left) and the global conservation of wave momentum $J = \psi \nabla \psi^*$ (right) for an ensemble of $32$ members, generated by $32$ independent realisation of Brownian motion in the NLS equation with stochastic dispersion \eqref{eq:stoch NLS}.}
    \label{fig:stoch dispersion 2}
\end{figure}



\section{Conclusion and future work}\label{sec:conclusion}
In this work, we have extended the variational formulation of multisymplectic PDEs to include stochastic perturbations.  Through a stochastic variational principle, we have derived stochastic multisymplectic PDEs and demonstrated their key properties. In Section \ref{subsec:stoch multi}, we established that these stochastic PDEs possess stochastic analogues of fundamental conservation laws, including the $1$-form quasi-conservation law \eqref{eq:stochastic 1-form conservation law}, the $2$-form structural conservation law \eqref{eq:stoch structural conservation laws}, and conservation laws arising from continuous variational symmetries via Noether's theorem \eqref{eq:stoch cls noether}. Examples of these stochastic multisymplectic systems were presented in Sections \ref{sec:SALT} and \ref{subsec:NLS}, where we derived the multisymplectic formulation of stochastic geometric transport based on inverse maps and stochastic NLS equations with stochastic transport and dispersion, respectively.

From a numerical perspective, we extended deterministic multisymplectic collocation methods to the stochastic setting, enabling their application to the stochastic multisymplectic PDEs derived in Section \ref{sec:stochastic multisymplectic cont}. The general class of methods, defined by \eqref{eq:stochastic RK} and \eqref{eq:stochastic RK Ham}, was shown to satisfy the discrete $2$-form conservation law under symplecticity conditions \eqref{eq:stochastic multisymplectic cond 1} and \eqref{eq:stochastic multisymplectic cond 2}, as proven in Theorem \ref{thm: multi symplectic form thm}. Additionally, the discrete momentum conservation law \eqref{eq:discrete momentum conserversion law} was verified for linear systems. Finally, numerical results of the stochastic multisymplectic integrators were presented in Section \ref{subsec:NLS numerics} for the stochastic NLS equations introduced in Section \eqref{subsec:NLS}, demonstrating the practical effectiveness of the proposed methods. 

\paragraph{Open problems and future work.} Following this work, there are several open problems that need to be addressed in the future.
\begin{itemize}
    \item The variational bicomplex (e.g., \cite{Anderson1992,bocharov1999symmetries,Olver1993})   provides the natural framework for studying multisymplectic  PDEs, as it has the clear separation of differential forms between independent and dependent variables \cite{BHL2010}. By using the discrete (finite difference) variational bicomplex, multisymplectic integrators  can be constructed to simulate multisymplectic PDEs through variational integration, see, e.g., \cite{peng2013,PH2023}. The introduction of stochasticity into variational bicomplex is a natural  extension of this work for the study of stochastic multisymplectic PDEs.
    \item Consider energy-preserving stochastic perturbations for multisymplectic systems. As the stochastic variational principle in this work introduces implicit time dependence in the form of white noise, we are interested in exploring stochasticity that preserves energy conservation, similar to that discussed in \cite{HH2021a}.
    \item Extend the driving class of stochasicity to geometric rough paths. Following \cite{CHLN2022}, a rough path driven variational principle can be used to derive rough multisymplectic PDEs. Using appropriate discretisations of the driving rough path, corresponding rough multisymplectic numerical methods can also be developed.
    \item Since many integrable PDE systems admit a multisymplectic structure, it is natural to investigate the types of stochastic perturbations that can be introduced while preserving integrability.For instance, how can the bi-Hamiltonian formalism and Lax pair formalism be extended to stochastic integrable systems?
\end{itemize}

\subsection*{Acknowledgments}
We wish to thank C. Cotter, D. Holm, O. Street and J. Woodfield for several thoughtful suggestions during the course of this work, which have improved or clarified the interpretation of its results.
RH is grateful for the support by the Office of Naval Research (ONR) grant award N00014-22-1-2082, Stochastic Parameterization of Ocean Turbulence for Observational Networks. LP is partially supported by JSPS KAKENHI (24K06852), JST CREST (JPMJCR1914, JPMJCR24Q5), and Keio University (Academic Development Fund, Fukuzawa Fund). RH thanks Keio University for their hospitality, where a large portion of the manuscript was prepared.

\bibliographystyle{abbrv}
\bibliography{main.bib}

\end{document}